\documentclass{article}

\usepackage{tikz}
\usepackage{tikz-cd}
\usepackage{url, hyperref}

\tikzset{->-/.style={decoration={
  markings,
  mark=at position .45 with {\arrow{>}}},postaction={decorate}}}

\usepackage{amsmath}
\usepackage{amssymb}
\usepackage{enumitem}
\usepackage{graphicx}
\usepackage{mathdots}
\usepackage{color}
\usepackage{diagbox}
\usepackage{array, makecell}
\usepackage{rotating}

\def\Z{\mathbb{Z}}

\def\qed{{\hfill $\Diamond$}}
\def\b1{{\bf 1}}

\def\g{\mathrm{g}}

\usepackage{amsthm}
\newtheorem{definition}{Definition}
\newtheorem{theorem}[definition]{Theorem}

\newtheorem{proposition}[definition]{Proposition}
\newtheorem{corollary}[definition]{Corollary}
\newtheorem{lemma}[definition]{Lemma}

\title{Tevelev degrees and Hurwitz moduli spaces}
\author{A. Cela, R. Pandharipande, J. Schmitt}
\date{April 2021}

\usepackage{graphicx}

\begin{document}

\maketitle
\begin{abstract}
We interpret the degrees which arise in Tevelev's study of scattering amplitudes in terms of  moduli spaces of Hurwitz covers. Via excess intersection
theory, the boundary geometry of the Hurwitz moduli space
yields a simple recursion for the Tevelev degrees 
(together with their natural two parameter generalization). We find exact solutions which specialize to Tevelev's
formula in his cases and connect to the projective geometry of  lines and Castelnuovo's classical count of $g^1_d$'s in other cases. For almost all values, the calculation of the
two parameter generalization of the Tevelev degree is new. A related count
of refined Dyck paths is solved along the way.
\end{abstract}

\tableofcontents

\section{Introduction}
\subsection{Hurwitz numbers}
The counting of Hurwitz covers of $\mathbb{P}^1$ with
fixed degree and branching data has been studied for more than
a century \cite{Hur}. The question is connected to complex
geometry, topology, and the representation theory of the 
symmetric group. In the last few decades, there has been an
resurgence of interest in Hurwitz covers motivated by
connections to Gromov-Witten theory and the geometry of the 
moduli space of curves, see \cite{BM,Dijk,ELSV, FabP,HM,OPcc,OPmm}.

The moduli space $\mathcal{H}_{g,d,n}$
parameterizes Hurwitz covers
$$ \pi: C \rightarrow \mathbb{P}^1$$
where $C$ is a complete, nonsingular, irreducible curve of genus $g$ with $n$
distinct markings $p_1,\ldots,p_n\in C$, the map $\pi$ is of degree $d$ with
$2g+2d-2$ distinct simple ramification points $$q_1, \ldots,q_{2g+2d-2}\in C \,,$$
 and all the points
$$\pi(p_1),\ldots, \pi(p_n),\pi(q_1),\ldots, \pi(q_{2g+2d-2})\in \mathbb{P}^1$$
are distinct. There is a natural compactification
$$ \mathcal{H}_{g,d,n} \subset \overline{\mathcal{H}}_{g,d,n}$$
by admissible covers \cite{HM}. 

The moduli space of admissible covers has 
two canonical maps determined by the domain and range of the cover:
\begin{equation*}
\begin{tikzcd}
& \overline{\mathcal{H}}_{g,d,n} \arrow[dr,"\epsilon_0"] \arrow[dl, "\epsilon_g", swap]& \\
\overline{\mathcal{M}}_{g,2g+2d-2+n}& &\overline{\mathcal{M}}_{0, 2g+2d-2+n}
\end{tikzcd}
\end{equation*}
The Hurwitz number of \cite{Hur}, in connected form,  is defined using the 
degree of the second map, 
$$\mathsf{Hur}_{g,d}=\frac{\text{deg}(\epsilon_0)}{d^n}\, .$$
The denominator $d^{n}$ removes the dependence
on $n$.{\footnote{The standard definition of the
Hurwitz number is for the $n=0$ geometry.}}

\subsection{Tevelev degrees}\label{Tevint}
A different degree was introduced by Tevelev in his study of
scattering amplitudes of stable curves \cite{Tev} motivated in part by \cite{ABC}.
Tevelev's degree, from the point of
view of the moduli spaces of Hurwitz covers, is defined as follows.
Consider the map 
$$\tau_{g,d,n}: \overline{\mathcal{H}}_{g,d,n} \rightarrow
\overline{\mathcal{M}}_{g,n}\times 
\overline{\mathcal{M}}_{0,n}$$
defined by
$\epsilon_g$ and $\epsilon_0$ 
forgetting the ramification points $q_j$ of the domain of the
cover and the branch points
$\pi(q_j)$ of the range. When
$$d=g+1\ \ \text{and}\ \  n=g+3\,,$$
both the domain and range of $\tau_{g,g+1,g+3}$ have dimension $5g$.
The number of forgotten ramification points on the domain is $4g$.
Tevelev's degree is
$$
\mathsf{Tev}_g= \frac{\text{deg}(\tau_{g,g+1,g+3})}{(4g)!}\, .
$$
The denominator $(4g)!$ reflects the possible orderings of the forgotten 
ramification points.
One of the central results of \cite{Tev} is the remarkably simple formula
\begin{equation}\label{gtt4}
\mathsf{Tev}_g= 2^g
\end{equation}
for all $g\geq 0$.
Tevelev provides several paths to the proof of \eqref{gtt4} involving beautiful aspects of
the 
classical geometry of curves.{\footnote{The
translation of Tevelev's definition to ours requires a change
of language. See  Section \ref{Tevd} for a discussion.}

Our goal here is to study Tevelev's degrees using the boundary
geometry of the moduli space of Hurwitz covers. To start, we observe the dimension
constraint required for the existence of a degree holds more generally.
For $\ell \in \mathbb{Z}$, let
$$d[g,\ell]=g+1+\ell     \  \ \text{and}\ \ \   n[g,\ell]=g+3+2\ell\, .$$
Consider the associated $\tau$-map (again forgetting ramification and branch points),
$$\tau_{g,\ell}: \overline{\mathcal{H}}_{g,d[g,\ell],n[g,\ell]} \rightarrow
\overline{\mathcal{M}}_{g,n[\g,\ell]}\times 
\overline{\mathcal{M}}_{0,n[g,\ell]}\, .$$
The domain and range of $\tau_{g,\ell}$ are both
of dimension $5g+4\ell$. We can therefore define
$$\mathsf{Tev}_{g,\ell} = \frac{\text{deg}(\tau_{g,\ell})}{(2g+2d[g,\ell]-2)!}\, ,$$
where we once again divide by the possible orderings of the forgotten ramification points.

Our first result is that the simplicity of the degree that
Tevelev found for $\ell=0$ continues to
hold for all non-negative $\ell$.

\begin{theorem} 
\label{vvtt3}
For all $g\geq0$ and $\ell\geq 0$, we have
$$\mathsf{Tev}_{g,\ell} = 2^g\, .$$
\end{theorem}

The proof of Theorem \ref{vvtt3} involves a recursion which
arises from the analysis
of the excess intersection theory of a particular fiber of $\tau_{g,\ell}$. 
The study of another naturally related degree 
is necessary for the argument.

For $1\leq r \leq d$, let
$\overline{\mathcal{H}}_{g,d,n,r}$
be the moduli space of admissible covers with $n$ markings {\em of which
$r$ lie in the same fiber of the cover}. See Figure \ref{Fig:Hbargdnr} for an illustration.

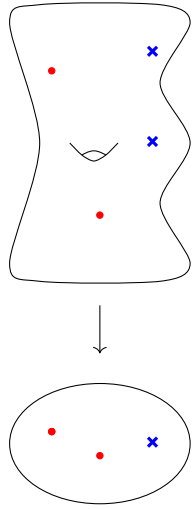
\begin{figure}[htb]
\centering
\begin{tikzpicture}[scale=0.8]
\draw plot [smooth cycle, tension = 0.5] coordinates {(0,0) (0.5, 0.3) (2.5,0.3) (3,0) (2.5,-1) (3,-2) (2.5,-3) (3,-4) (2.5,-4.3) (0.5,-4.3) (0,-4) (0.5,-2)};

\draw plot [smooth, tension = 0.5] coordinates {(1,-2) (1.2,-2.2) (1.4,-2.3) (1.6,-2.2) (1.8,-2)};
\draw plot [smooth, tension = 0.8] coordinates { (1.2,-2.2) (1.4,-2.13) (1.6,-2.2)};

\draw [->] (1.5, -4.7) -- (1.5, -5.5);
\draw plot [smooth cycle, tension = 1] coordinates {(0,-7) (1.5, -6) (3,-7) (1.5,-8)};

\draw[blue, very  thick] (2.3,-6.9) -- (2.45, -7.05);
\draw[blue, very  thick] (2.3,-7.05) -- (2.45, -6.9);
\draw[blue, very  thick, yshift=5cm] (2.3,-6.9) -- (2.45, -7.05);
\draw[blue, very  thick, yshift=5cm] (2.3,-7.05) -- (2.45, -6.9);
\draw[blue, very  thick, yshift=6.5cm] (2.3,-6.9) -- (2.45, -7.05);
\draw[blue, very  thick, yshift=6.5cm] (2.3,-7.05) -- (2.45, -6.9);

\filldraw[red]  (1.5, -7.2) circle (1.5pt);
\filldraw[red, yshift=4cm]  (1.5, -7.2) circle (1.5pt);
\filldraw[red]  (0.7, -6.8) circle (1.5pt);
\filldraw[red]  (0.7, -6.8) circle (1.5pt);
\filldraw[red, yshift=6cm]  (0.7, -6.8) circle (1.5pt);

\end{tikzpicture}
\caption{A curve $C \to \mathbb{P}^1$ in $\overline{\mathcal{H}}_{1,d,4,2}$ with $n=4$ markings in $C$ of which $r=2$ lie in the fibre over the same point of $\mathbb{P}^1$.}
\label{Fig:Hbargdnr}
\end{figure}

\noindent Since $r$ markings lie in the same fiber, the range of the
map $\epsilon_0$ is altered,
$$ \epsilon_0: \overline{\mathcal{H}}_{g,d,n,r} \rightarrow
\overline{\mathcal{M}}_{0, 2g+2d-2+n-r+1}\, .$$
For $\ell\in \mathbb{Z}$, let the $\tau$-map
$$\tau_{g,\ell,r}: \overline{\mathcal{H}}_{g,d[g,\ell],n[g,\ell],r} \rightarrow
\overline{\mathcal{M}}_{g,n[\g,\ell]}\times 
\overline{\mathcal{M}}_{0,n[g,\ell]-r+1}\, .$$
be obtained as before by forgetting all the ramification points of the domain
and branch points of the range of the cover.
The domain and range of $\tau_{g,\ell,r}$ are both
of dimension $5g+4\ell-r+1$. The degrees
$$\mathsf{Tev}_{g,\ell,r} = 
\frac{\text{deg}(\tau_{g,\ell,r})}{(2g+2d[g,\ell]-2)!}\, .$$
appear in the proof of Theorem \ref{vvtt3}}. 

In case $r=1$, we recover the previously defined Tevelev degree
$$ \mathsf{Tev}_{g,\ell,1} =\mathsf{Tev}_{g,\ell} \, .$$
In case $\ell=0$ and $r=1$,
we recover Tevelev's original count
$$\mathsf{Tev}_{g,0,1}=\mathsf{Tev}_g\, .$$
We therefore have a two parameter 
variation of the Tevelev degree.

The study of $\mathsf{Tev}_{g,\ell,r}$ separates into
two main cases depending upon whether $\ell$ is non-negative
or non-positive.

\subsection{Results for \texorpdfstring{$\ell \geq 0$}{l>=0}}
The following two results completely determine
the degrees $\mathsf{Tev}_{g,\ell,r}$ in all cases where $\ell\geq 0$.

\begin{theorem} 
\label{vvtt6} For all $g\geq0$, $\ell= 0$, and $1\leq r \leq g+1$, we have
$$\mathsf{Tev}_{g,0,r} = 2^g -\sum_{i=0}^{r-2} \binom{g}{i}\, .$$
\end{theorem}

\begin{theorem} 
\label{vvtt7}
For all $g\geq0$, $\ell> 0$, and $1\leq r \leq g+1+\ell$, we have:
\begin{eqnarray*}
\mathsf{Tev}_{g,\ell,r} &=&  2^g \ \ \ \ \ \ \ \ \ \ \ \ \ \ \text{if}
\ \ \ell\geq r\, ,\\
\mathsf{Tev}_{g,\ell,r} &=& \mathsf{Tev}_{g,0,r-\ell} \ \ \ \ \ \text{if}
\ \ \ell<r\, .
\end{eqnarray*}
\end{theorem}

We have already seen that the $\ell=0$ and $r=1$ case 
recovers Tevelev's result. If we take $r=g+1+\ell$, then
Theorems \ref{vvtt6} and \ref{vvtt7} yield
\begin{equation} \label{vssw}
\mathsf{Tev}_{g,\ell,g+1+\ell} =1\, 
\end{equation}
for $\ell\geq 0$.
For a Hurwitz covering  
$$[\pi:C\rightarrow \mathbb{P}^1]\in {\mathcal{H}}_{g,d[g,\ell],n[g,\ell], r=d[g,\ell]}\, ,$$
a full fiber of $\pi$ is specified by the $r$ markings.
The $r$ markings determine a line bundle $L$ on $C$. The evaluation
\eqref{vssw} can also be obtained by studying a classical
transformation 
related to the associated
complete linear series $\mathbb{P}(H^0(C,L))$. 
See Section \ref{Sect:Gale} for a discussion.

\subsection{Results for \texorpdfstring{$\ell\leq 0$}{l<=0}}

\label{dd13}

The calculation of $\mathsf{Tev}_{g,\ell,r}$ for $\ell\leq 0$ has a 
more
intricate structure. As before, we require
\begin{equation}\label{ffrd}
1\leq r \leq d[g,\ell] = g+1+\ell\, .
\end{equation}
Since the range of $\tau_{g,\ell,r}$ is
$$\overline{\mathcal{M}}_{g,n[\g,\ell]}\times 
\overline{\mathcal{M}}_{0,n[g,\ell]-r+1}\, ,$$
we also require
\begin{equation} \label{ttg5}
    n[g,\ell]-r+1 = g+3+2\ell-r+1 \geq 3\, .
\end{equation}
If either \eqref{ffrd} or
\eqref{ttg5}
is not
satisfied, then $\mathsf{Tev}_{g,\ell,r} =0$ by definition.

In fact, $r\geq 1$ and condition \eqref{ttg5} can be written together as
\begin{equation}\label{ggt}
1\leq r \leq g+1+2\ell
\end{equation}
which implies the upper bound of \eqref{ffrd} when $\ell\leq 0$. 

Suppose $\ell\leq 0$ and $r\geq 1$ are both fixed. 
Let 
$$g[\ell,r]= r-2\ell-1$$
be the minimum genus permitted by \eqref{ggt}.
We express the
solutions to the associated Tevelev counts for all
genera as an infinite vector
$$\mathsf{T}_{\ell,r}=\big(\mathsf{Tev}_{g[\ell,r],\ell,r}\, ,\,
\mathsf{Tev}_{g[\ell,r]+1,\ell,r}\, ,\, \mathsf{Tev}_{g[\ell,r]+2,\ell,r}\, ,\,
\mathsf{Tev}_{g[\ell,r]+3,\ell,r}\, ,
\ldots\big)\, .$$
The $j^{th}$ component{\footnote{We start the vector index at 0, so
$\mathsf{T}_{\ell,r}= \big(\mathsf{T}_{\ell,r}[0],\mathsf{T}_{\ell,r}[1],
\mathsf{T}_{\ell,r}[2], \mathsf{T}_{\ell,r}[3],\ldots\big)$}}
of $\mathsf{T}_{\ell,r}$ is
$$\mathsf{T}_{\ell,r}[j] = \mathsf{Tev}_{g[\ell,r]+j,\ell,r}\, .$$
We will write a formula for $\mathsf{T}_{\ell,r}$.

Define the infinite vector $\mathsf{E}_s$ for $s\geq 1$ as 
$$\mathsf{E}_s=\left(2^{s-1}-\sum_{i=0}^{s-2}\binom{s-1}{i}\, ,\  
2^{s}-\sum_{i=0}^{s-2}\binom{s}{i}\, ,\ 
2^{s+1}-\sum_{i=0}^{s-2}\binom{s+1}{i}\, ,\ \ldots\right)\, .$$
The $j^{th}$ component{\footnote{Again, the vector index is started at 0.}}
of $\mathsf{E}_{s}$ is
$$\mathsf{E}_{s}[j] = 2^{s+j-1} - \sum_{i=0}^{s-2}\binom{s+j-1}{i}\, .$$
The first few vectors are
$$
\begin{array}{cccccccc}
\mathsf{E}_1 & = & \big(& 1, &2,&4,& 8,& 16,\ \ldots\ \big)\\
\mathsf{E}_2 & = & \big(& 1,& 3,&7,& 15,& 31,\ \ldots\ \big)\\
\mathsf{E}_3 & = & \big(& 1,& 4,& 11,& 26,& 57,\  \ldots\ \big)
\end{array}
$$
The leading components are always 1,
$$\mathsf{E}_s[0]=1\, .$$
Moreover, the components of the full set of vectors $\mathsf{E}_s$ is easily seen
to satisfy a Pascal-type addition law
$$\mathsf{E}_{s+1}[j+1] = \mathsf{E}_{s+1}[j] +\mathsf{E}_{s}[j+1]\, .$$
By Theorem \ref{vvtt6}, we have
$\mathsf{E}_s[j]= \mathsf{Tev}_{s+j-1,0,s}$,
so 
\begin{equation}\label{dggee}
\mathsf{T}_{0,r} =\mathsf{E}_r\, .
\end{equation}

Our first result for $\ell\leq 0$ expresses $\mathsf{T}_{\ell,r}$ as a
finite linear combination of the vectors $\mathsf{E}_s$.
The coefficients are non-negative integers and
are given by a simple path counting formula.

The path counting occurs on the integer lattice $\mathbb{Z}^2$.
We are only interested in the lattice points
\begin{equation}
\mathcal{A}=\{ \, (\ell,r)\, | \, \ell\leq 0, r \geq 1 \, \} \subset \mathbb{Z}^2\, .
\end{equation}
Our paths start at the point $(0,1)$, must stay on the lattice points of
$\mathcal{A}$ and take steps only by the vectors
$$\mathsf{U}=(0,1) 
\ \ \ {\text{and}} \ \ \ \mathsf{D}= (-1,-1)\,.$$
Let $\mathsf{P}(\ell,r)$ be the set of such paths from $(0,1)$ to
$(\ell,r)$, see Figure 2.

Let $\gamma\in \mathsf{P}(\ell,r)$.
The {\em index} $\mathsf{Ind}(\gamma)$ is the number of points of
$\gamma$ which meet the
boundary $\partial \mathcal{A}$,
$$\partial \mathcal{A}=  \{ \, (\ell,1)\, | \, \ell\leq 0\, \} \ \cup\
\{ \, (0,r)\, | \,  r \geq 1 \, \}\, \subset \mathcal{A}\, .$$

\begin{theorem}\label{n33r}
Let $\ell\leq 0$ and $r\geq 1$. Then,
$$\mathsf{T}_{\ell,r} = \sum_{\gamma\in \mathsf{P}(\ell,r)} 
\mathsf{E}_{\mathsf{Ind}(\gamma)}\, .$$
\end{theorem}

\begin{figure}
\centering
\begin{tikzpicture}[xscale=0.4,yscale=0.4]

\draw [help lines] (0,1) grid (-4, 5);
\draw [thick] (0,1) to (0,2) to (-1,1) to (-1,2) to (-2,1) to (-2,2) to (-3,1);

\draw [help lines] (6,1) grid (2, 5);
\draw [thick] (6,1) to (6,2) to (5,1) to (5,2) to (5,3) to (4,2) to (3,1);

\draw [help lines] (12,1) grid (8, 5);
\draw [thick] (12,1) to (12,3) to (10,1) to (10,2) to (9,1);

\draw [help lines] (18,1) grid (14, 5);
\draw [thick] (18,1) to (18,4) to (15,1);

\draw [help lines] (24,1) grid (20, 5);
\draw [thick] (24,1) to (24,3) to (23,2) to (23,3) to (21,1);

\end{tikzpicture}
\caption{Possible paths from $(0,1)$ to $(-3,1)$.  There are exactly $C_3 = 5$
possible paths, where $C_3$ is the third Catalan number. Our paths are equivalent to Dyck paths \cite{Stan}.}
\end{figure}
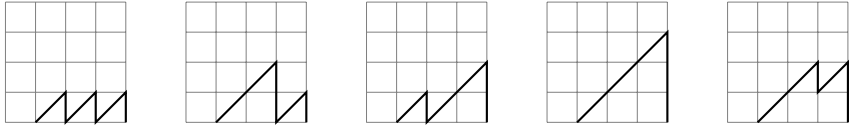

 Theorem \ref{n33r} takes the simplest form in case
 $(\ell,r)=(0,r)$. The set
$\mathsf{P}(0,r)$ contains a unique path{\footnote{The unique path in the
$(\ell,1)=(0,1)$ case is degenerate and consists only of the single point
$(0,1)$.}}
of index $r$. Theorem \ref{n33r} then just recovers \eqref{dggee}.

For all $\ell\leq 0$ and $r\geq 1$,
the set $\mathsf{P}(\ell,r)$ is finite. Moreover, the
index of a path $\gamma\in \mathsf{P}(\ell,r)$ is bounded by
$$\mathsf{Ind}(\gamma) \leq r-\ell+1\, .$$
The following result is then a consequence of Theorem \ref{n33r}.

\begin{corollary} \label{Cor:TZcombination}
Let $\ell\leq 0$ and $r\geq 1$. Then, 
$\mathsf{T}_{\ell,r}$ lies in the $\mathbb{Z}_{\geq 0}$-linear span of
the vectors
$\mathsf{E}_1,\ldots, \mathsf{E}_{r-\ell+1}$.
\end{corollary}

\noindent The  first examples{\footnote{A table of values of $\mathsf{T}_{\ell,r}$
can be found in Section \ref{f665}.}} with $\ell<0$ and $r=1$ are:
$$
\begin{array}{llll}
\mathsf{T}_{-1,1} = \mathsf{E}_3\, , & \mathsf{T}_{-2,1} = 2\mathsf{E}_4\, , & 
\mathsf{T}_{-3,1} = 2\mathsf{E}_4+3\mathsf{E}_5\, , &
\mathsf{T}_{-4,1} =4\mathsf{E}_4+6\mathsf{E}_5+4\mathsf{E}_6 \, .
\end{array}
$$
In fact,
if $\ell<0$, then $\mathsf{T}_{\ell,r}$ lies in the span of
$\mathsf{E}_3,\ldots, \mathsf{E}_{r-\ell+1}$
since every path in $\mathsf{P}(\ell,r)$ has index at least 3.

The leading coefficient of $\mathsf{T}_{\ell,1}$ for $\ell\leq 0$ has an alternative
geometric interpretation: 
$\mathsf{T}_{\ell,1}[0]$ is the number of $g^1_{|\ell|+1}$'s carried
by a general curve of genus $2|\ell|$.
The above examples\footnote{Remember $\mathsf{E}_s[0]=1$.} show
$$
\begin{array}{cccc}
\mathsf{T}_{-1,1}[0] = 1\, , & \mathsf{T}_{-2,1}[0] =2\,, 
&  \mathsf{T}_{-3,1}[0] = 5\,, & 
\mathsf{T}_{-4,1}[0] = 14\,,
\\
\end{array}
$$
which are recognizable as Catalan numbers. 
By the well-known counting of Dyck paths \cite{Deut,Stan},
$$\mathsf{T}_{\ell,1}[0]=|\mathsf{P}(\ell,1)| = \frac{1}{|\ell|+1}
\binom{|2\ell|}{|\ell|}\, ,$$
which, of course, agrees with Castelnuovo's
famous count of linear series \cite{Cast}.

The coefficients of 
the expansion of $\mathsf{T}_{\ell,r}$ in terms of 
$\mathsf{E}_s$ determined by Theorem  \ref{n33r} may be viewed
as a refined Dyck path counting 
problem.\footnote{See \cite{Deut} for the study of various refined Dyck path counting problems.} An exact solution in terms
of binomials is presented in Section \ref{rdpc}. For example,
the coefficients of the expansion
$$\mathsf{T}_{\ell,1}= \sum_{s=3}^{-\ell+2} c^s_{\ell,1} \mathsf{E}_s$$
for $\ell<-1$ are{\footnote{The coefficient formula for $\ell=-1$ and $s=3$ is
also valid if the $0$'s are cancelled.}} 
$$c^s_{\ell,1} 
= \frac{(s-2)(s-3)}{|\ell| -1} {{|2 \ell| -s} \choose {|\ell| + 2- s} } \, . $$

In fact, the summation over paths in Theorem \ref{n33r} can be 
exactly evaluated to yield closed formulas for the Tevelev degrees.

\begin{theorem}\label{negcase}
The Tevelev degrees for  $\ell \leq 0$ and $g\geq r-2\ell-1$ are

\vspace{5pt}
\noindent $\bullet$ for $g\geq 0$ and $r = 1$, 
\begin{equation*}
    \mathsf{T}_{g,\ell,1} = 2^g -2 \sum_{i=0}^{-\ell-2} \binom{g}{i} + (-\ell -2) \binom{g}{-\ell-1} + \ell \binom{g}{-\ell}\, ,
\end{equation*}

\vspace{5pt}
\noindent $\bullet$
for $g\geq 0$ and $r>1$,
\begin{equation*}
    \mathsf{T}_{g,\ell,r} = 2^g -2 \sum_{i=0}^{-\ell-2} \binom{g}{i} + (-\ell + r -3) \binom{g}{-\ell-1} + (\ell-1) \binom{g}{-\ell} - \sum_{i=-\ell+1}^{r-\ell-2} \binom{g}{i}\,.
\end{equation*}
\end{theorem}

\vspace{5pt}
The $r=1$ formula of Theorem \ref{negcase} 
can be viewed as
the $r=1$ specialization{\footnote{Care has to be taken for the meaning of $r=1$ specialization of the
last sum.}} of the
$r>1$ formula. 
In fact, the formulas for $\ell\geq 0$  of Theorems \ref{vvtt3}-\ref{vvtt7} can also be seen
as specializations of the $r>1$ formula  of Theorem \ref{negcase}. 
The $r>1$ formula
of Theorem \ref{negcase} therefore represents a complete calculation of the Tevelev
degrees $\mathsf{T}_{g,\ell,r}$. A unified point of view for
all $\ell$ is presented in Section \ref{fccp}.

\subsection{Hurwitz cycles}
Both Hurwitz numbers and Tevelev degrees are
aspects of a more general question.
Via the diagram
\begin{equation*}
\begin{tikzcd}
& \overline{\mathcal{H}}_{g,d,n} \arrow[dr,"\epsilon_0"] \arrow[dl, "\epsilon_g", swap]& \\
\overline{\mathcal{M}}_{g,2g+2d-2+n}& &\overline{\mathcal{M}}_{0, 2g+2d-2+n}\ ,
\end{tikzcd}
\end{equation*}
we always have a map
$$\tau: \overline{\mathcal{H}}_{g,d,n}\rightarrow
\overline{\mathcal{M}}_{g,2g+2d-2+n} \times \overline{\mathcal{M}}_{0, 2g+2d-2+n}\, .$$
The dimension of the range of $\tau$ usually
exceeds the dimension of the domain. A basic question here is to compute the push-forward of the fundamental class:
\begin{equation}\label{qqq}
\tau_*[\overline{\mathcal{H}}_{g,d,n}] \in
\mathsf{CH}^*\big(\overline{\mathcal{M}}_{g,2g+2d-2+n} \times \overline{\mathcal{M}}_{0, 2g+2d-2+n}\big)\, .
\end{equation}
It follows from the main result of \cite{FabP} that
the push-forward is given by tautological classes\footnote{Formally, the results in \cite{FabP} only show that the push-forward to $\overline{\mathcal{M}}_{g,2g+2d-2+n}$ is tautological. However, using the cellular structure of $\overline{\mathcal{M}}_{2g+2d-2+n}$, the
Chow groups on the right hand side of \eqref{qqq} split as the tensor product of Chow groups of the factors. Combining the results of \cite{FabP} with techniques from \cite{Lian1}, each individual term in the tensor product decomposition of \eqref{qqq} can then be seen to be tautological.}
$$\tau_*[\overline{\mathcal{H}}_{g,d,n}] \in
\mathsf{R}^*\big(\overline{\mathcal{M}}_{g,2g+2d-2+n}\big) \otimes \mathsf{R}^*\big(\overline{\mathcal{M}}_{0, 2g+2d-2+n}\big)\, ,$$
but very few complete formulas are known.{\footnote{There are several other natural variations. We could consider more complicated ramification profiles of the Hurwitz covers.
Another idea is to consider relative stable map spaces. The result along these lines which is closest to the spirit of question \eqref{qqq} is Pixton's formula for the double ramification cycle \cite{JPPZ}.
For further aspects of the study of Hurwitz moduli spaces
and their associated cycles on the moduli spaces
of curves, see \cite{BP,Lian2,Lian1,P,PHHT, SvZ, Zelm}.}} 
Hurwitz numbers and Tevelev degrees give
slivers of cohomological information about
the push-forward \eqref{qqq}.

\subsection*{Acknowledgments}
We thank J. Tevelev for a wonderful
lecture in the {\em Algebraic Geometry and Moduli Seminar} at ETH Z\"urich in February 2021. Our paper began by studying his degrees from the perspective of 
Hurwitz covers. We are grateful to G. Farkas, C. Lian, and
D. Ranganathan for related discussions.

We have used computer experiments with the software SageMath \cite{Sage} to calculate examples and explore formulas for $\mathsf{Tev}_{g,\ell,r}$. The code can be found \href{https://cocalc.com/share/6888accf93758e2d1a7c76d0d030f415bced300d/Tevelev\%20degrees.ipynb?viewer=share}{\textcolor{blue}{here}} or can be obtained from the authors upon request.

A.C. was supported by SNF-200020-182181.
R.P. was supported by SNF-200020-182181,  ERC-2017-AdG-786580-MACI, and SwissMAP.  J.S. was supported by the SNF Early Postdoc Mobility grant 184245 and thanks
the Max Planck Institute for Mathematics in Bonn for its hospitality. 

The project has received funding
from the European Research Council (ERC) under the European Union Horizon 2020 research and innovation program (grant agreement No 786580).


\section{Tevelev degrees}\label{Tevd}
Tevelev \cite{Tev} arrives at the numbers $\mathsf{Tev}_g$ from a slightly different point of view. Below we describe his perspective and how it relates to definitions presented in Section \ref{Tevint}.

To start, let $g \geq 0$ and $n=g+3$. Fix a general curve 
$$(C,p_1, \ldots, p_n) \in \mathcal{M}_{g,n}\, .$$
Then a line bundle $\mathcal{L}$ on $C$ of degree $d=g+1$ satisfies
\[\chi(C, \mathcal{L}) = h^0(C,\mathcal{L}) - h^1(C, \mathcal{L}) = g+1 +1-g = 2\]
by Riemann-Roch. 
For a general such line bundle $\mathcal{L} \in \mathrm{Pic}^{g+1}(C)$, the first cohomology vanishes, so  $\mathcal{L}$ has precisely two sections. For general $\mathcal{L}$, the two sections do not have a common zero and therefore define a degree $d$ map
\[
\varphi_{\mathcal{L}} : C \to \mathbb{P}^1 = \mathbb{P}(H^0(C,\mathcal{L})^\vee)\, .
\]

Tevelev \cite{Tev} constructs a rational \emph{scattering amplitude map} $\Lambda$,
\begin{equation} \label{eqn:Lambdadef}
    \Lambda : \mathrm{Pic}^{g+1}(C) \dashrightarrow \mathcal{M}_{0,n}\, ,\ \ \mathcal{L} \mapsto (\mathbb{P}^1, \varphi_{\mathcal L}(p_1), \ldots, \varphi_{\mathcal L}(p_n))\, .
\end{equation}
In \cite[Definition 1.7]{Tev}, $\Lambda$ is shown to be 
dominant and generically finite, and  $\mathsf{Tev}_g$ is defined to be the degree of  $\Lambda$.  The result
$$\mathsf{Tev}_g=2^g$$
is \cite[Theorem 1.14]{Tev}.

\def\Pic{\mathcal{P}ic \,}

To relate Tevelev's definition to the definition of 
Section \ref{Tevint}, let $\Pic_{g,n}^d$ be the \emph{universal Picard stack} over $\mathcal{M}_{g,n}$ parameterizing the data
\[
(C, p_1, \ldots, p_n, \mathcal{L})
\]
of a nonsingular, $n$-pointed genus $g$ curve $C$ together with a line bundle $\mathcal{L}$ of degree $d=g+1$. The morphism
\begin{align*}
    \mathcal{H}_{g,d,n} &\to \Pic_{g,n}^d\,,\\
    \left(\pi : (C,p_1, \ldots, p_n) \to \mathbb{P}^1\right) &\mapsto (C,p_1, \ldots, p_n, \mathcal{L}= \pi^* \mathcal{O}_{\mathbb{P}^1}(1))
\end{align*}
is dominant and generically finite of degree $b!$, where
\[
b = 2g-2+2d = 4g
\]
is the number of (simple) ramification points of a Hurwitz cover $\pi$. 
Indeed, for $\mathcal{L} \in \Pic_{g,n}^d$ general, the fiber will be the map $\pi = \varphi_{\mathcal L}$ together with the $b!$ many choices of ordering the $b$ simple ramification points.

To conclude, consider the map
\[
\tau_{g,d,n} : \mathcal{H}_{g,d,n} \to \mathcal{M}_{g,n} \times \mathcal{M}_{0,n}
\]
from Section \ref{Tevint}. For $(C,p_1, \ldots, p_n) \in \mathcal{M}_{g,n}$ general and every $\mathcal{L}$ in the fibre of the map $\Lambda$ of \eqref{eqn:Lambdadef} over a general point $(\mathbb{P}^1, q_1, \ldots, q_n) \in \mathcal{M}_{0,n}$, we have precisely $b!$ points in the fibre
\[
\tau_{g,d,n}^{-1}( (C,p_1, \ldots, p_n), (\mathbb{P}^1, q_1, \ldots, q_n) )\,.
\]
Hence, we obtain
\[
\mathsf{Tev}_g = \frac{\deg(\tau_{g,d,n})}{(4g)!}
\]
as defined in Section \ref{Tevint}.

\section{Recursion via fiber geometry}

The basic property satisfied by the Tevelev degree which is used in the
proofs of Theorems \ref{vvtt3}-\ref{n33r} is the following recursion.

\begin{proposition}\label{recc}
Let $g, r \geq 1$. Let $\ell \in \mathbb{Z}$ satisfy 
 $$n[g,\ell] - r+1=g+3+2\ell-r+1 \ \geq \ 3\, .$$
Then, we have the recursion
\begin{equation} \label{eqn:tevrecursion}
    \mathsf{Tev}_{g,\ell,r} = \mathsf{Tev}_{g-1,\ell,\max(1,r-1)} + \mathsf{Tev}_{g-1,\ell+1,r+1}\,.
\end{equation}
\end{proposition}
\begin{proof}
Up to a combinatorial factor, $\mathsf{Tev}_{g,\ell,r}$ is defined as the degree of the map
$$\tau_{g,\ell,r}: \overline{\mathcal{H}}_{g,d[g,\ell],n[g,\ell],r} \rightarrow
\overline{\mathcal{M}}_{g,n[\g,\ell]}\times 
\overline{\mathcal{M}}_{0,n[g,\ell]-r+1}\, .$$
To prove the recursion \eqref{eqn:tevrecursion}, we compute the
degree by analysing the fiber over 
a particular point $$(C,D)\in \overline{\mathcal{M}}_{g,n[\g,\ell]}\times 
\overline{\mathcal{M}}_{0,n[g,\ell]-r+1}\, .$$
The degree of $\tau_{g,\ell,r}$ is simply 
the degree of the zero cycle 
$$\tau_{g,\ell,r}^* [(C,D)]
\in \mathsf{CH}_0\left(\overline{\mathcal{H}}_{g,d[g,\ell],n[g,\ell],r} \right)\, .$$
The actual fiber $\tau_{g,\ell,r}^{-1} [(C,D)]$ will have excess
dimension, so some care must be taken in the analysis.

\vspace{6pt}
\noindent $\bullet$ The pointed curve $C$ in $\overline{\mathcal{M}}_{g,n[\g,\ell]}$ is chosen
to have the following form:
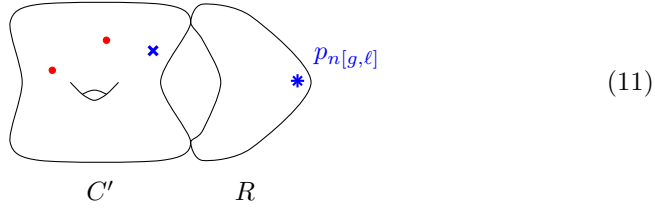
\begin{equation} \label{eqn:Ccurveshape}
\begin{tikzpicture}[baseline = -0.9 cm, scale=0.8]
\draw plot [smooth cycle, tension = 0.5] coordinates {(0,0) (0.5, 0.3) (2.5,0.3) (3,0) (2.5,-1) (3,-2) (2.5,-2.3) (0.5,-2.3) (0,-2) (0.2,-1)};
\draw plot [smooth cycle, tension = 0.5] coordinates {(3.2,-0.3) (3,0) (3.2,0.3) (4,0.1) (5,-1) (4,-2.1) (3.2,-2.3) (3,-2) (3.2,-1.8) (3.5,-1)};

\draw plot [smooth, tension = 0.5, yshift=1cm] coordinates {(1,-2) (1.2,-2.2) (1.4,-2.3) (1.6,-2.2) (1.8,-2)};
\draw plot [smooth, tension = 0.8, yshift=1cm] coordinates { (1.2,-2.2) (1.4,-2.13) (1.6,-2.2)};

\draw[blue, very  thick, yshift=6.5cm] (2.3,-6.9) -- (2.45, -7.05);
\draw[blue, very  thick, yshift=6.5cm] (2.3,-7.05) -- (2.45, -6.9);
\draw[blue, thick, yshift=6cm, xshift=2.4cm] (2.3,-6.9) -- (2.45, -7.05);
\draw[blue, thick, yshift=6cm, xshift=2.4cm] (2.3,-7.05) -- (2.45, -6.9);
\draw[blue, thick, yshift=6cm, xshift=2.4cm] (2.265,-6.975) -- (2.485, -6.975);
\draw[blue, thick, yshift=6cm, xshift=2.4cm] (2.375,-7.085) -- (2.375, -6.865);

\filldraw[red, yshift=6cm]  (0.7, -6.8) circle (1.5pt);
\filldraw[red, yshift=6.5cm, xshift=0.9cm]  (0.7, -6.8) circle (1.5pt);

\draw (1.5, -2.8) node {$C'$};
\draw (3.9, -2.8) node {$R$};
\draw[blue] (5.6, -0.6) node {$p_{n[g,\ell]}$};
\end{tikzpicture}
\end{equation}
 The curve $C$ consists of a general curve $C' \in \overline{\mathcal{M}}_{g-1,n[\g,\ell]-1+2}$ attached at two points to a rational curve $R$. All the markings lie on $C'$ except for 
 the last marking $p_{n[g,\ell]}$ which lies on $R$.
 The marking $p_{n[g,\ell]}$ is the last of the distinguished subset of
 $r$ markings. 
 
 \vspace{6pt}
 \noindent $\bullet$ The pointed curve $D$ in $\overline{\mathcal{M}}_{0,n[g,\ell]-r+1}$  is any general point (in particular, the curve $D$ is nonsingular).
\vspace{6pt}

To compute the cycle $\tau_{g,\ell,r}^* [(C,D)]$, let
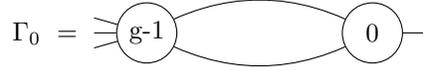
\[
\Gamma_0 \ = \ 
\begin{tikzpicture}[baseline = -0.1 cm]
\draw (0,0) to [bend left] (3,0);
\draw (0,0)  to [bend right] (3,0);
\draw (3,0) -- (3.7,0);
\draw (0,0) -- (-0.7,0);
\draw (0,0) -- (-0.7,0.2);
\draw (0,0) -- (-0.7,-0.2);
\filldraw[white] (0,0) circle (0.4cm) node {g-1};
\filldraw[white] (3,0) circle (0.4cm) node {0};
\draw (0,0) circle (0.4cm) node {g-1};
\draw (3,0) circle (0.4cm) node {0};
\end{tikzpicture}
\]
be the stable graph of the curve \eqref{eqn:Ccurveshape},
and let
\[
\xi_{\Gamma_0} : \overline{\mathcal{M}}_{\Gamma_0} = \overline{\mathcal{M}}_{g-1,n[\g,\ell]-1+2} \times \overline{\mathcal{M}}_{0,3} \to \overline{\mathcal{M}}_{g,n[\g,\ell]}
\]
be the associated gluing map, sending the point $(C',R) \in \overline{\mathcal{M}}_{\Gamma_0}$ to the curve $C$. 
We consider the following commutative diagram:
\begin{equation} \label{eqn:Cfibrediagram}
\begin{tikzcd}
 \coprod \overline{\mathcal{H}}_{(\Gamma, \Gamma')} \arrow[d] \arrow[r]& \overline{\mathcal{H}}_{g,d[g,\ell],n[g,\ell],r}\arrow[d]\\
 \coprod \overline{\mathcal{M}}_{\widehat \Gamma_0} \arrow[d] \arrow[r,"\xi_{\widehat \Gamma_0}"] & \overline{\mathcal{M}}_{g,n[\g,\ell] + b[g,\ell]} \arrow[d]\\
 \overline{\mathcal{M}}_{\Gamma_0} \arrow[r,"\xi_{\Gamma_0}"]& \overline{\mathcal{M}}_{g,n[\g,\ell]}
\end{tikzcd}
\end{equation}
The upper right vertical arrow 
is the map remembering the domain of the admissible cover together with the $n[g,\ell]$ markings and the 
\[
b[g,\ell] = 2 g - 2 + 2 d[g,\ell] = 4 g + 2 \ell
\]
simple ramification points. The bottom right vertical arrow
is the map forgetting these extra markings, so the composition on the right is the familiar map $\epsilon_g$.
The disjoint union on the left side of the diagram is over all stable graphs $\widehat{\Gamma}_0$ that can be obtained by distributing $b[g,\ell]$ legs to the two vertices of $\Gamma_0$. The lower square in the diagram is then not quite Cartesian, but the disjoint union of the $\overline{\mathcal{M}}_{\widehat \Gamma_0}$
maps properly and birationally to the corresponding fibre product. This property will be sufficient for the intersection theoretic computations below (see  \cite[Lemma A.18]{BaeSchmitt} for a formal argument).
 
The upper square is a fibre diagram on the level of sets, as proven in \cite[Proposition 3.2]{Lian2}. Here, the disjoint union is over boundary strata $\overline{\mathcal{H}}_{(\Gamma, \Gamma')}$ of $\overline{\mathcal{H}}_{g,d[g,\ell],n[g,\ell],r}$ parameterizing maps from a curve with stable graph $\Gamma$ mapping to a curve of stable graph $\Gamma'$. The index set of the disjoint union consists of isomorphism classes of $\widehat{\Gamma}_0$-structures\footnote{See \cite[Appendix A]{GraberPandharipande} for a discussion of the notion of an $\widehat{\Gamma}_0$-structure on a stable graph $\Gamma$.} on $\Gamma$ such that the induced map $$E(\widehat{\Gamma}_0) \to E(\Gamma) \to E(\Gamma')$$ is surjective. The lift of the upper fibre diagram to the level of schemes is described in \cite[Proposition 3.3]{Lian2}.

Our strategy is now as follows. To compute
\[
\tau_{g,\ell,r}^* [(C,D)] = \epsilon_g^* [C] \cdot \epsilon_0^* [D]\, ,
\]
we consider $[C] = (\xi_{\Gamma_0})_*[(C',R)]$ and use the diagram \eqref{eqn:Cfibrediagram} to decompose $\epsilon_g^* [C]$ into terms supported on the spaces $\overline{\mathcal{H}}_{(\Gamma, \Gamma')}$. Only three types of contributions survive after taking the product with $\epsilon_0^* [D]$. 
We begin by describing  these three surviving cases and
computing their contributions. Multiplicities
and excess intersections
appear.

\vspace{10pt}
\noindent \textbf{Contribution 1.} 
The Hurwitz cover degenerates as indicated in diagram \eqref{eqn:Ccurveshape1} with the degrees of the map written in green.\footnote{In the picture \eqref{eqn:Ccurveshape1}, the preimage of the rightmost component
of the bottom curve consists of a single rational curve mapping with degree $2$
and $d-2$ rational curves mapping with degree 1. Here, and
in the following pictures, we only draw one representative of such rational
components mapping with degree $1$.}
The last $r$ markings are in blue, and the
last marking is starred.
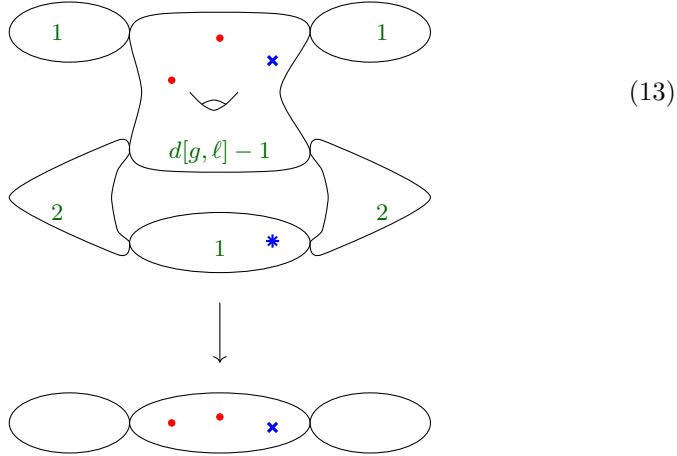
\begin{equation} \label{eqn:Ccurveshape1}
\begin{tikzpicture}[baseline = -0.9 cm, scale=0.8]
\draw plot [smooth cycle, tension = 0.5] coordinates {(0,0) (0.5, 0.3) (2.5,0.3) (3,0) (2.5,-1) (3,-2) (2.5,-2.3) (0.5,-2.3) (0,-2) (0.2,-1)};
\draw plot [smooth cycle, tension = 1] coordinates {(0,-3.5) (1.5,-3) (3,-3.5) (1.5,-4) };

\draw plot [smooth cycle, tension = 1, yshift=3.5cm] coordinates {(3,-3.5) (4,-3) (5,-3.5) (4,-4)  };
\draw plot [smooth cycle, tension = 1, yshift=3.5cm, xshift=-5cm] coordinates {(3,-3.5) (4,-3) (5,-3.5) (4,-4)  };

\draw plot [smooth cycle, tension = 0.4] coordinates { (3.2,-2.3) (3,-2) (3.2,-1.8) (5, -2.75) (3.2,-3.7) (3,-3.5) (3.2,-3.3) (3.3,-2.75)};
\begin{scope}[yscale=1,xscale=-1, xshift=-3cm]
  \draw plot [smooth cycle, tension = 0.4] coordinates { (3.2,-2.3) (3,-2) (3.2,-1.8) (5, -2.75) (3.2,-3.7) (3,-3.5) (3.2,-3.3) (3.3,-2.75)};
\end{scope}

\draw plot [smooth, tension = 0.5, yshift=1cm] coordinates {(1,-2) (1.2,-2.2) (1.4,-2.3) (1.6,-2.2) (1.8,-2)};
\draw plot [smooth, tension = 0.8, yshift=1cm] coordinates { (1.2,-2.2) (1.4,-2.13) (1.6,-2.2)};

\draw[blue, very  thick, yshift=6.5cm] (2.3,-6.9) -- (2.45, -7.05);
\draw[blue, very  thick, yshift=6.5cm] (2.3,-7.05) -- (2.45, -6.9);

\draw[blue, thick, yshift=3.5cm] (2.3,-6.9) -- (2.45, -7.05);
\draw[blue, thick, yshift=3.5cm] (2.3,-7.05) -- (2.45, -6.9);
\draw[blue, thick, yshift=3.5cm] (2.265,-6.975) -- (2.485, -6.975);
\draw[blue, thick, yshift=3.5cm] (2.375,-7.085) -- (2.375, -6.865);

\draw[blue, very  thick, yshift=0.4cm] (2.3,-6.9) -- (2.45, -7.05);
\draw[blue, very  thick, yshift=0.4cm] (2.3,-7.05) -- (2.45, -6.9);


\filldraw[red, yshift=6cm]  (0.7, -6.8) circle (1.5pt);
\filldraw[red, yshift=6.7cm, xshift=0.8cm]  (0.7, -6.8) circle (1.5pt);
\filldraw[red, yshift=0.3cm]  (0.7, -6.8) circle (1.5pt);
\filldraw[red, yshift=0.4cm, xshift=0.8cm]  (0.7, -6.8) circle (1.5pt);

\draw [->] (1.5, -4.5) -- (1.5, -5.5);

\draw plot [smooth cycle, tension = 1, yshift=-3cm] coordinates {(0,-3.5) (1.5,-3) (3,-3.5) (1.5,-4) };
\draw plot [smooth cycle, tension = 1, yshift=-3cm] coordinates {(3,-3.5) (4,-3) (5,-3.5) (4,-4)  };
\draw plot [smooth cycle, tension = 1, yshift=-3cm, xshift=-5cm] coordinates {(3,-3.5) (4,-3) (5,-3.5) (4,-4)  };

\draw[black!60!green] (1.5, -2) node {\small $d[g,\ell]-1$};
\draw[black!60!green] (1.5, -3.6) node {\small $1$};
\draw[black!60!green] (-1.2, -3) node {\small $2$};
\draw[black!60!green] (4.2, -3) node {\small $2$};
\draw[black!60!green] (-1.2, 0) node {\small $1$};
\draw[black!60!green] (4.2, 0) node {\small $1$};
\end{tikzpicture}
\end{equation}
Each of the two rational components mapping with degree $2$ must carry two of the ramification points.
The total number of loci of the form \eqref{eqn:Ccurveshape1} is therefore given by
\begin{equation} \label{eqn:numloci1}
    \frac{1}{2} \binom{b[d,\ell]}{2,2,b[d,\ell]-4} = \frac{1}{8} \frac{b[d,\ell]!}{(b[d,\ell]-4)!} \,,
\end{equation}
with the factor $1/2$ corresponding to the fact that the two pairs of branch points going to the outer components can be exchanged without changing the locus. 

For $r>1$, the locus is parameterized by the space
\begin{equation} \label{eqn:parloci1}
    \overline{\mathcal{H}}_{(\Gamma, \Gamma')} = \overline{\mathcal{H}}_{g-1,d[g,\ell]-1,n[g,\ell]-1+2,r-1} \times\underbrace{ \overline{\mathcal{H}}_{0,2,2,2} \times \overline{\mathcal{H}}_{0,2,2,2}}_{=\{\mathrm{pt}\}}\,.
\end{equation}
The unique point of $\overline{\mathcal{H}}_{0,2,2,2}$ is given by the map
\[
(\mathbb{P}^1, 1, -1) \xrightarrow{z \mapsto z^2} \mathbb{P}^1\, ,
\]
and the cover has no remaining automorphisms. 

On the other hand, for $r=1$, the position of marking $p_{n[g,\ell]}$ is no longer determined by the map and markings on the genus $g-1$ component. If $r=1$, let
\[
\widehat{\epsilon}_0 : \overline{\mathcal{H}}_{g-1,d[g,\ell]-1,n[g,\ell]-1+2,r-1} \to \overline{\mathcal{M}}_{0,4g+2\ell-4+n[g,\ell]-1}
\]
be the map remembering both the $4g+2\ell-4$ branch points and the images of the $n[g,\ell]-1$ markings on the target of the cover. Then the locus fits into a diagram
\begin{equation} \label{eqn:diagramhelp}
\begin{tikzcd}
 \overline{\mathcal{H}}_{(\Gamma, \Gamma')} \arrow[r] \arrow[d] & \overline{\mathcal{M}}_{0,4g+2\ell-4+n[g,\ell]} \arrow[d]\\
 \overline{\mathcal{H}}_{g-1,d[g,\ell]-1,n[g,\ell]-1+2,1} \arrow[r] & \overline{\mathcal{M}}_{0,4g+2\ell-4+n[g,\ell]-1}
\end{tikzcd}
\end{equation}
where the right vertical arrow  is the map forgetting the last marking.

Since, independent of $r$, the map \eqref{eqn:Ccurveshape1} is unramified at all the nodes, we see from \cite[Proposition 3.3]{Lian2} that the locus $\overline{\mathcal{H}}_{(\Gamma, \Gamma')}$ appears with multiplicity 1 in the upper fibre diagram \eqref{eqn:Cfibrediagram}. On the other hand, there are $8$ different $\widehat{\Gamma}_0$-structures on the graph $\Gamma$ of the domain curve in \eqref{eqn:Ccurveshape1} that satisfy the assumptions of \cite[Proposition 3.2]{Lian2}. Indeed, the image of the first edge of $\widehat{\Gamma}_0$ can be chosen freely among the $4$ nonseparating edges of $\Gamma$, and then the second edge must go to one of the $2$ edges on the other side. So overall, in the disjoint union in the top left of the diagram \eqref{eqn:Cfibrediagram}, there are 
\[
\frac{b[d,\ell]!}{(b[d,\ell]-4)!}
\]
many components that can contribute.

By the strategy explained above, we see that the contribution to the degree of $\tau_{g,\ell,r}^* [(C,D)]$ coming from each of the components
is given by the degree of the map
\begin{equation} \label{eqn:degreehelp}
\overline{\mathcal{H}}_{(\Gamma, \Gamma')} \to \underbrace{\overline{\mathcal{M}}_{g-1,n[g,\ell]+1}}_{=\overline{\mathcal{M}}_{\Gamma_0}} \times \overline{\mathcal{M}}_{0,n[g,\ell]-r+1}
\end{equation}

\noindent $\bullet$ For $r>1$, the locus $\overline{\mathcal{H}}_{(\Gamma, \Gamma')}$ is given by \eqref{eqn:parloci1} and the above map fits as the top horizontal arrow into a fibre diagram
\[
\begin{tikzcd}
 \overline{\mathcal{H}}_{g-1,d[g,\ell]-1,n[g,\ell]+1,r-1} \arrow[r] \arrow[d] & \overline{\mathcal{M}}_{g-1,n[g,\ell]+1} \times \overline{\mathcal{M}}_{0,n[g,\ell]-r+1} \arrow[d]\\
 \overline{\mathcal{H}}_{g-1,d[g,\ell]-1,n[g,\ell]-1,r-1} \arrow[r]  & \overline{\mathcal{M}}_{g-1,n[g,\ell]-1} \times \overline{\mathcal{M}}_{0,n[g,\ell]-r+1}
\end{tikzcd}
\]
So the degree of the top map is the same as the degree of the bottom map. Using 
\[d[g,\ell]-1 = d[g-1,\ell]\, , \ \ \ n[g,\ell]-1 = n[g-1,\ell]\, ,\]
the bottom map is precisely $\tau_{g-1,\ell,r-1}$.

\vspace{5pt} \noindent $\bullet$
For $r=1$, we must distinguish two more cases:

\vspace{5pt}
\noindent $\ \ \ \diamond$ For $n[g,\ell] \geq 4$, the map \eqref{eqn:degreehelp} appears as the top horizontal arrow in the fibre diagram
\[
\begin{tikzcd}
 \overline{\mathcal{H}}_{(\Gamma, \Gamma')} \arrow[r] \arrow[d] & \overline{\mathcal{M}}_{g-1,n[g,\ell]+1} \times \overline{\mathcal{M}}_{0,n[g,\ell]} \arrow[d]\\
 \overline{\mathcal{H}}_{g-1,d[g,\ell]-1,n[g,\ell]-1,1} \arrow[r]  & \overline{\mathcal{M}}_{g-1,n[g,\ell]-1} \times \overline{\mathcal{M}}_{0,n[g,\ell]-1}
\end{tikzcd}
\]
and so the degree of \eqref{eqn:degreehelp} is given by the degree of the bottom map $\tau_{g-1,\ell,1}$. 

\vspace{5pt}
\noindent $\ \ \ \diamond$
For $n[g,\ell]=3$, the map \eqref{eqn:degreehelp} factors through the left vertical map in \eqref{eqn:diagramhelp}. Since the vertical map drops dimension by $1$, we know that the morphism \eqref{eqn:degreehelp} can no longer be dominant and thus yields a zero contribution. The vanishing here is compatible with our original definition, since  $\mathsf{Tev}_{g-1,\ell,1}$ 
is defined to be zero since
$n[g-1,\ell]=2$.
\vspace{5pt}

Overall, we see that the total part of the degree arising from Contribution 1 is given by
\[
\frac{b[d,\ell]!}{(b[d,\ell]-4)!} \deg \tau_{g-1,\ell,\max(1,r-1)} = b[d,\ell]! \cdot  \mathsf{Tev}_{g-1,\ell,\max(1,r-1)}\,.
\]

\vspace{10pt}
\noindent \textbf{Contribution 2.} 
The second type of locus that can contribute parameterizes covers of the form illustrated in diagram \eqref{eqn:Ccurveshape2}.
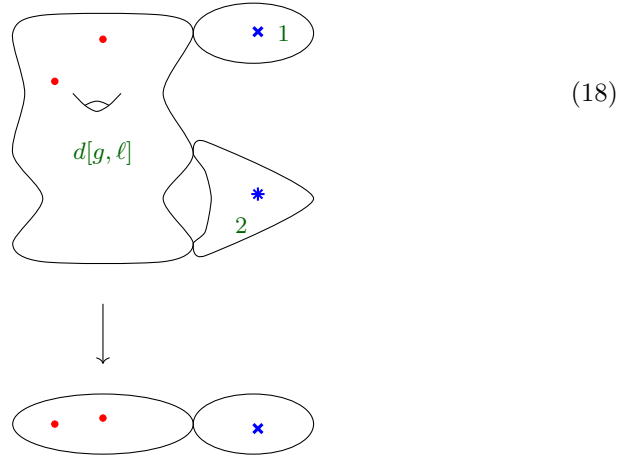
\begin{equation} \label{eqn:Ccurveshape2}
\begin{tikzpicture}[baseline = -0.9 cm, scale=0.8]
\draw plot [smooth cycle, tension = 0.5] coordinates {(0,0) (0.5, 0.3) (2.5,0.3) (3,0) (2.5,-1) (3,-2) (2.5,-2.75) (3,-3.5) (2.5,-3.8) (0.5,-3.8) (0,-3.5) (0.5,-2.75) (0,-2) (0.2,-1)};

\draw plot [smooth cycle, tension = 1, yshift=3.5cm] coordinates {(3,-3.5) (4,-3) (5,-3.5) (4,-4)  };

\draw plot [smooth cycle, tension = 0.4] coordinates { (3.2,-2.3) (3,-2) (3.2,-1.8) (5, -2.75) (3.2,-3.7) (3,-3.5) (3.2,-3.3) (3.3,-2.75)};

\draw plot [smooth, tension = 0.5, yshift=1cm] coordinates {(1,-2) (1.2,-2.2) (1.4,-2.3) (1.6,-2.2) (1.8,-2)};
\draw plot [smooth, tension = 0.8, yshift=1cm] coordinates { (1.2,-2.2) (1.4,-2.13) (1.6,-2.2)};

\draw[blue, very  thick, yshift=7cm, xshift=1.7cm] (2.3,-6.9) -- (2.45, -7.05);
\draw[blue, very  thick, yshift=7cm, xshift=1.7cm] (2.3,-7.05) -- (2.45, -6.9);

\draw[blue, thick, yshift=4.3cm, xshift=1.7cm] (2.3,-6.9) -- (2.45, -7.05);
\draw[blue, thick, yshift=4.3cm, xshift=1.7cm] (2.3,-7.05) -- (2.45, -6.9);
\draw[blue, thick, yshift=4.3cm, xshift=1.7cm] (2.265,-6.975) -- (2.485, -6.975);
\draw[blue, thick, yshift=4.3cm, xshift=1.7cm] (2.375,-7.085) -- (2.375, -6.865);


\draw[blue, very  thick, yshift=0.4cm, xshift=1.7cm] (2.3,-6.9) -- (2.45, -7.05);
\draw[blue, very  thick, yshift=0.4cm, xshift=1.7cm] (2.3,-7.05) -- (2.45, -6.9);

\filldraw[red, yshift=6cm]  (0.7, -6.8) circle (1.5pt);
\filldraw[red, yshift=6.7cm, xshift=0.8cm]  (0.7, -6.8) circle (1.5pt);
\filldraw[red, yshift=0.3cm]  (0.7, -6.8) circle (1.5pt);
\filldraw[red, yshift=0.4cm, xshift=0.8cm]  (0.7, -6.8) circle (1.5pt);

\draw [->] (1.5, -4.5) -- (1.5, -5.5);

\draw plot [smooth cycle, tension = 1, yshift=-3cm] coordinates {(0,-3.5) (1.5,-3) (3,-3.5) (1.5,-4) };
\draw plot [smooth cycle, tension = 1, yshift=-3cm] coordinates {(3,-3.5) (4,-3) (5,-3.5) (4,-4)  };
\draw[black!60!green] (1.5, -2) node {\small $d[g,\ell]$};
\draw[black!60!green] (3.8, -3.2) node {\small $2$};
\draw[black!60!green] (4.5, 0) node {\small $1$};
\end{tikzpicture}
\end{equation}
The rational component mapping with degree $2$ must carry two of the ramification points. The total number of loci of the form \eqref{eqn:Ccurveshape2} is given by
\begin{equation} \label{eqn:numloci2}
    \binom{b[d,\ell]}{2} = \frac{1}{2} \frac{b[d,\ell]!}{(b[d,\ell]-2)!} \,.
\end{equation}
The locus $\overline{\mathcal{H}}_{(\Gamma, \Gamma')}$ is parameterized by the space
\begin{equation} \label{eqn:parloci2}
    \overline{\mathcal{H}}_{g-1,d[g,\ell],n[g,\ell]+1,r+1} \times  \overline{\mathcal{H}}_{0,2,3,2}\,.
\end{equation}
The $r+1$ index in \eqref{eqn:parloci2} is
explained by the following basic
important observation: all $r+1$ nodes of the genus $g-1$ domain
component connecting to rational curves containing  blue markings map to the same point (which is the node of the range curve). 

The space $\overline{\mathcal{H}}_{0,2,3,2}$, parameterizes double covers
\begin{equation} \label{eqn:H0232element}
    f : (\mathbb{P}^1, p, z, z', q_1, q_2) \to \mathbb{P}^1
\end{equation}
such that $f(z)=f(z')$ and such that $q_1, q_2$ are ramification points of the cover. The natural map 
\[\delta: \overline{\mathcal{H}}_{0,2,3,2} \to \overline{\mathcal{M}}_{0,4}\]
sending the point \eqref{eqn:H0232element} to $$(\mathbb{P}^1, f(p), f(z), f(q_1), f(q_2)) \in \overline{\mathcal{M}}_{0,4}$$ has degree $2$, corresponding to the two choices $z,z'$ of preimage of $f(z)$.

By the same argument as before, the loci
\eqref{eqn:Ccurveshape2}
appear with multiplicity 1 in the upper fibre diagram \eqref{eqn:Cfibrediagram}. On the other hand, there are $2$ $\widehat{\Gamma}_0$-structures on the graph $\Gamma$ of the domain curve in \eqref{eqn:Ccurveshape2}, corresponding to the two choices of sending the edges of $\widehat{\Gamma}_0$ to the two nonseparating edges $e,e'$ of $\Gamma$. However, these two $\widehat{\Gamma}_0$-structures are isomorphic (since the cover $\Gamma \to \Gamma'$ has an automorphism exchanging the two edges $e,e'$).  So overall, in the disjoint union in the top left of the diagram \eqref{eqn:Cfibrediagram} there are 
\[
\frac{1}{2} \cdot \frac{b[d,\ell]!}{(b[d,\ell]-2)!}
\]
many components that can contribute.

However,  a new phenomenon (not seen
in Contribution 1) appears here: a dimension count shows that the loci \eqref{eqn:parloci2} only have codimension $1$ in their ambient space, whereas the gluing map $\xi_{\Gamma_0}$ had codimension $2$. The excess class is given by
\[
- \psi_{w} \otimes 1 - 1 \otimes \psi_{z}
\]
where $w,z$ are the preimages of one of the nonseparating nodes in \eqref{eqn:Ccurveshape2}. 
The class $\psi_{z}$ on $\overline{\mathcal{H}}_{0,2,3,2}$ is the pullback of $\psi_2$ on $\overline{\mathcal{M}}_{0,4}$ under the map $\delta$ above. We calculate
\[
\int_{\overline{\mathcal{H}}_{0,2,3,2}} \psi_{z} = \deg(\delta) \cdot \int_{\overline{\mathcal{M}}_{0,4}} \psi_2 = 2 \cdot 1 = 2\,.
\]

The contribution to the degree of $\tau_{g,\ell,r}^* [(C,D)]$ coming from each locus above is given by the degree of the excess class capped with the preimage of a point under the map
\[
\overline{\mathcal{H}}_{g-1,d[g,\ell],n[g,\ell]+1,r+1} \times  \overline{\mathcal{H}}_{0,2,3,2} \to \overline{\mathcal{M}}_{g-1,n[g,\ell]+1} \times \overline{\mathcal{M}}_{0,n[g,\ell]-r+1}\, .
\]
Since the above map does not depend on the factor $\overline{\mathcal{H}}_{0,2,3,2}$ in the domain, the only nonzero contribution can come from the term $-1 \otimes \psi_{z}$. After eliminating the factor $\overline{\mathcal{H}}_{0,2,3,2}$ by
integrating $-1 \otimes \psi_{z}$
(and obtaining a factor $-2$ from the degree of $-\psi_z$), we are left with the degree of $\tau_{g-1,\ell+1,r+1}$. 

The total part of the degree arising from Contribution 2 is given by
\[
 \frac{1}{2} \cdot \frac{b[d,\ell]!}{(b[d,\ell]-2)!} \cdot (-2) \cdot \deg \tau_{g-1,\ell+1,r+1} = - b[d,\ell]! \cdot  \mathsf{Tev}_{g-1,\ell+1,r+1}\,.
\]

\vspace{10pt}
\noindent \textbf{Contribution 3.} 
The third type of locus that can contribute parameterizes covers of the form illustrated in diagram \eqref{eqn:Ccurveshape3}.
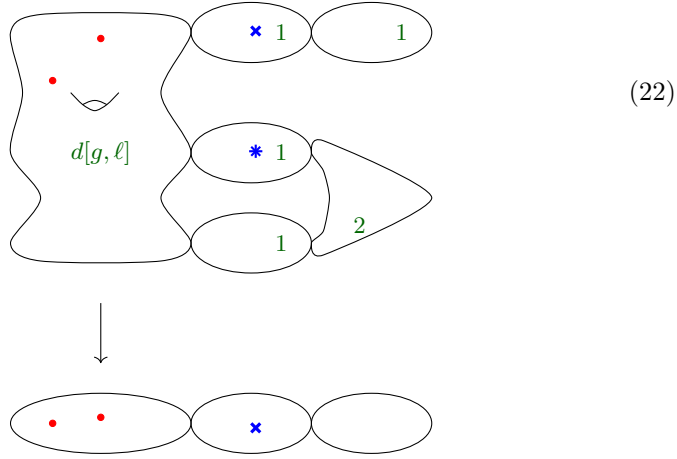
\begin{equation} \label{eqn:Ccurveshape3}
\begin{tikzpicture}[baseline = -0.9 cm, scale=0.8]
\draw plot [smooth cycle, tension = 0.5] coordinates {(0,0) (0.5, 0.3) (2.5,0.3) (3,0) (2.5,-1) (3,-2) (2.5,-2.75) (3,-3.5) (2.5,-3.8) (0.5,-3.8) (0,-3.5) (0.5,-2.75) (0,-2) (0.2,-1)};

\draw plot [smooth cycle, tension = 1, yshift=3.5cm] coordinates {(3,-3.5) (4,-3) (5,-3.5) (4,-4)  };
\draw plot [smooth cycle, tension = 1, yshift=3.5cm, xshift=2cm] coordinates {(3,-3.5) (4,-3) (5,-3.5) (4,-4)  };
\draw plot [smooth cycle, tension = 1, yshift=1.5cm] coordinates {(3,-3.5) (4,-3) (5,-3.5) (4,-4)  };
\draw plot [smooth cycle, tension = 1, yshift=0cm] coordinates {(3,-3.5) (4,-3) (5,-3.5) (4,-4)  };

\draw plot [smooth cycle, tension = 0.4, xshift=2cm] coordinates { (3.2,-2.3) (3,-2) (3.2,-1.8) (5, -2.75) (3.2,-3.7) (3,-3.5) (3.2,-3.3) (3.3,-2.75)};

\draw plot [smooth, tension = 0.5, yshift=1cm] coordinates {(1,-2) (1.2,-2.2) (1.4,-2.3) (1.6,-2.2) (1.8,-2)};
\draw plot [smooth, tension = 0.8, yshift=1cm] coordinates { (1.2,-2.2) (1.4,-2.13) (1.6,-2.2)};

\draw[blue, very  thick, yshift=7cm, xshift=1.7cm] (2.3,-6.9) -- (2.45, -7.05);
\draw[blue, very  thick, yshift=7cm, xshift=1.7cm] (2.3,-7.05) -- (2.45, -6.9);

\draw[blue, thick, yshift=5cm, xshift=1.7cm] (2.3,-6.9) -- (2.45, -7.05);
\draw[blue, thick, yshift=5cm, xshift=1.7cm] (2.3,-7.05) -- (2.45, -6.9);
\draw[blue, thick, yshift=5cm, xshift=1.7cm] (2.265,-6.975) -- (2.485, -6.975);
\draw[blue, thick, yshift=5cm, xshift=1.7cm] (2.375,-7.085) -- (2.375, -6.865);

\draw[blue, very  thick, yshift=0.4cm, xshift=1.7cm] (2.3,-6.9) -- (2.45, -7.05);
\draw[blue, very  thick, yshift=0.4cm, xshift=1.7cm] (2.3,-7.05) -- (2.45, -6.9);


\filldraw[red, yshift=6cm]  (0.7, -6.8) circle (1.5pt);
\filldraw[red, yshift=6.7cm, xshift=0.8cm]  (0.7, -6.8) circle (1.5pt);
\filldraw[red, yshift=0.3cm]  (0.7, -6.8) circle (1.5pt);
\filldraw[red, yshift=0.4cm, xshift=0.8cm]  (0.7, -6.8) circle (1.5pt);

\draw [->] (1.5, -4.5) -- (1.5, -5.5);

\draw plot [smooth cycle, tension = 1, yshift=-3cm] coordinates {(0,-3.5) (1.5,-3) (3,-3.5) (1.5,-4) };
\draw plot [smooth cycle, tension = 1, yshift=-3cm] coordinates {(3,-3.5) (4,-3) (5,-3.5) (4,-4)  };
\draw plot [smooth cycle, tension = 1, yshift=-3cm, xshift=2cm] coordinates {(3,-3.5) (4,-3) (5,-3.5) (4,-4)  };
\draw[black!60!green] (1.5, -2) node {\small $d[g,\ell]$};
\draw[black!60!green, xshift=2cm] (3.8, -3.2) node {\small $2$};
\draw[black!60!green] (4.5, 0) node {\small $1$};
\draw[black!60!green] (6.5, 0) node {\small $1$};
\draw[black!60!green] (4.5, -2) node {\small $1$};
\draw[black!60!green] (4.5, -3.5) node {\small $1$};
\end{tikzpicture}
\end{equation}
The rational component mapping with degree $2$ must carry two of the ramification points.
The total number of loci of the form \eqref{eqn:Ccurveshape2} is given by
\begin{equation} \label{eqn:numloci3}
    \binom{b[d,\ell]}{2} = \frac{1}{2} \frac{b[d,\ell]!}{(b[d,\ell]-2)!} \,.
\end{equation}
The locus $\overline{\mathcal{H}}_{(\Gamma, \Gamma')}$ is parameterized by the space
\begin{equation} \label{eqn:parloci3}
    \overline{\mathcal{H}}_{g-1,d[g,\ell],n[g,\ell]+1,r+1} \times  \underbrace{\overline{\mathcal{H}}_{0,2,2,2}}_{=\{\mathrm{pt}\}}\,.
\end{equation}
Again the multiplicity is $1$, but there are $4$ non-isomorphic $\widehat{\Gamma}_0$-structures on the graph $\Gamma$:
one of the two edges of $\widehat{\Gamma}_0$ must go to the edge of $\Gamma$ connecting the genus $g-1$ vertex to the vertex containing the last marking. The other can go to either one of the two edges incident to the vertex with the degree $2$ map. So overall, in the disjoint union in the top left of the diagram \eqref{eqn:Cfibrediagram} there are
\[
 2 \frac{b[d,\ell]!}{(b[d,\ell]-2)!}
\]
components that can contribute.

By the strategy explained above,  the contribution to the degree of $\tau_{g,\ell,r}^* [(C,D)]$ coming from each locus above is given by the degree of  the map
\[
\overline{\mathcal{H}}_{g-1,d[g,\ell],n[g,\ell]+1,r+1}  \to \overline{\mathcal{M}}_{g-1,n[g,\ell]+1} \times \overline{\mathcal{M}}_{0,n[g,\ell]-r+1}
\]
The total part of the degree arising from Contribution 3 is given by
\[
 \frac{1}{2} \cdot 2 \cdot \frac{b[d,\ell]!}{(b[d,\ell]-2)!} \cdot 2 \cdot \deg \tau_{g-1,\ell+1,r+1} =  2 b[d,\ell]! \cdot  \mathsf{Tev}_{g-1,\ell+1,r+1}\,.
\]

\vspace{10pt}
\noindent \textbf{Exclusion of other components.}
\vspace{10pt}

To conclude, we must prove that
Contributions 1, 2 and 3 are the only loci in the diagram \eqref{eqn:Cfibrediagram} that can contribute. 

Let $\overline{\mathcal{H}}_{(\Gamma, \Gamma')}$ be a locus in \eqref{eqn:Cfibrediagram} dominating the moduli space $$\overline{\mathcal{M}}_{g-1,n[g,\ell]+1} \times \overline{\mathcal{M}}_{0,n[g,\ell]-r+1}\,.$$
There must then be a vertex $v_{g-1} \in V(\Gamma)$ such that the corresponding factor of $\overline{\mathcal{H}}_{(\Gamma, \Gamma')}$ parameterizes covers from a genus $g-1$ curve. Since the $n[g,\ell]-1$ first markings will go to the vertex $v_{g-1}$ after forgetting the ramification points and stabilizing, their images in $\Gamma'$ will similarly stabilize to a unique vertex $v_0$. Therefore, in the cover $\Gamma \to \Gamma'$, the vertex $v_g$ maps to $v_0$. Next, we observe that the map 
\begin{equation} \label{eqn:dominantmap}
\overline{\mathcal{H}}_{(\Gamma, \Gamma')} \to \overline{\mathcal{M}}_{g-1,n[g,\ell]+1} \times \overline{\mathcal{M}}_{0,n[g,\ell]-r+1}
\end{equation}
only depends upon the factors of $\overline{\mathcal{H}}_{(\Gamma, \Gamma')}$ lying over the vertex $v_0$. Indeed, the map to $\overline{\mathcal{M}}_{g-1,n[g,\ell]+1}$ only depends on the factor for $v_{g-1}$. On the other hand, since the stabilization of $\Gamma'$, after forgetting the branch points, consists of the single vertex $v_0$, the factors of $\overline{\mathcal{H}}_{(\Gamma, \Gamma')}$ over $v_0$ determine the component of the map \eqref{eqn:dominantmap} to $\overline{\mathcal{M}}_{0,n[g,\ell]-r+1}$. 

Since the factors over $v_0$  form a finite cover of the moduli space $\overline{\mathcal{M}}_{0,n(v_0)}$ associated to $v_0$, the valence $n(v_0)$ must satisfy
\[
n(v_0) - 3 \geq \dim\ \overline{\mathcal{M}}_{g-1,n[g,\ell]+1} \times \overline{\mathcal{M}}_{0,n[g,\ell]-r+1} = b[g,\ell] + n[g,\ell] -r -4
\]
in order for the map \eqref{eqn:dominantmap} to be dominant. 

A short computation shows that the valence condition is only possible in one of the following cases:
\begin{enumerate}[label=\alph*)]
    \item[(i)] $\Gamma'$ has two other vertices $w_1, w_2$, both of valence $3$, 
    \item[(ii)] $\Gamma'$ has one other vertex $w$, of valence $4$, 
    \item[(iii)] $\Gamma'$ has one other vertex $w$, of valence $3$.
\end{enumerate}

We further observe that for a simple branch point at the vertex $v_0$ whose ramification point is \emph{not} on the component associated to $v_{g-1}$, the position of the branch point does not affect the image under the map \eqref{eqn:dominantmap}. Therefore, in cases (i) and (ii) all the simple ramification points over $v_0$ are located at the vertex $v_{g-1}$, whereas in case (iii) at most one of them can be on a separate component.

Finally, we observe that for a leaf of the graph $\Gamma'$ which is not equal to $v_0$ (so one of the vertices $w,w_1,w_2$ above), the leaf must contain at least two of the simple branch points. Indeed, by the condition that there is a $\widehat{\Gamma}_0$ structure on $\Gamma \to \Gamma'$, there must be a circular path from $v_{g-1}$ to itself in $\Gamma$ whose edges surject to the edges of $\Gamma'$. Thus over a leaf there must be a vertex with at least two edges adjacent, which needs at least two ramification points on that vertex.

\vspace{10pt}
\noindent \textbf{Case (i)} There are two possibilities $\Gamma'_1$
and $\Gamma_2'$ for the graph $\Gamma'$.
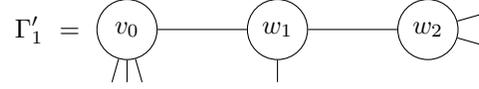
\[
\Gamma_1' \ = \ 
\begin{tikzpicture}[baseline = -0.1 cm]
\draw (0,0) to (4,0);
\draw (2,0) -- (2,-0.7);
\draw (0,0) -- (0,-0.7);
\draw (0,0) -- (0.2,-0.7);
\draw (0,0) -- (-0.2,-0.7);
\draw (4,0) -- (4.7,0.2);
\draw (4,0) -- (4.7,-0.2);
\filldraw[white] (0,0) circle (0.4cm) node {$v_0$};
\filldraw[white] (2,0) circle (0.4cm) node {$w_1$};
\filldraw[white] (4,0) circle (0.4cm) node {$w_2$};
\draw (0,0) circle (0.4cm) node {$v_0$};
\draw (2,0) circle (0.4cm) node {$w_1$};
\draw (4,0) circle (0.4cm) node {$w_2$};
\end{tikzpicture}
\]
By the general comments above, the two markings at $w_2$ must be branch points, and we see that there are at least $b[g,\ell]-3$ simple ramification points at $v_{g-1}$. The configuration is only possible if the degree at the vertex $v_{g-1}$ remains at $d[g,\ell]$, otherwise the number of ramification points would drop by at least $4$. The marking at $w_1$ is then forced to be the image of the marking $p_{n[g,\ell]}$, since otherwise this marking would be contained in the genus $g-1$ component of the curve, contradicting  the shape \eqref{eqn:Ccurveshape} we have fixed before. We conclude that only the pair $(\Gamma, \Gamma_1')$ of Contribution 3 has the specified shape.

\vspace{10pt}
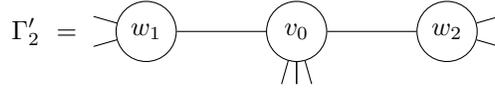
\[
\Gamma_2' \ = \ 
\begin{tikzpicture}[baseline = -0.1 cm]
\draw (0,0) to (4,0);
\draw (2,0) -- (2,-0.7);
\draw (2,0) -- (2,-0.7);
\draw (2,0) -- (2.2,-0.7);
\draw (2,0) -- (1.8,-0.7);
\draw (4,0) -- (4.7,0.2);
\draw (4,0) -- (4.7,-0.2);
\draw (0,0) -- (-0.7,0.2);
\draw (0,0) -- (-0.7,-0.2);
\filldraw[white] (0,0) circle (0.4cm) node {$v_0$};
\filldraw[white] (2,0) circle (0.4cm) node {$w_1$};
\filldraw[white] (4,0) circle (0.4cm) node {$w_2$};
\draw (0,0) circle (0.4cm) node {$w_1$};
\draw (2,0) circle (0.4cm) node {$v_0$};
\draw (4,0) circle (0.4cm) node {$w_2$};
\end{tikzpicture}
\]
The markings at $w_1, w_2$ belonging to branch points, so we have precisely $b[g,\ell]-4$ branch points at $v_0$. Then the component $v_{g-1}$ must have degree $d[g,\ell]-1$, so that there is an additional vertex of genus $0$ over $v_0$ mapping with degree $1$. We are then forced to be in the case of Contribution 1.

\vspace{10pt}
\noindent \textbf{Case (ii)} The shape of the graph $\Gamma'$ is as follows:
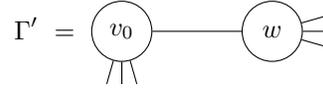
\[
\Gamma' \ = \ 
\begin{tikzpicture}[baseline = -0.1 cm]
\draw (2,0) to (4,0);
\draw (2,0) -- (2,-0.7);
\draw (2,0) -- (2,-0.7);
\draw (2,0) -- (2.2,-0.7);
\draw (2,0) -- (1.8,-0.7);
\draw (4,0) -- (4.7,0.2);
\draw (4,0) -- (4.7,-0.2);
\draw (4,0) -- (4.7,0);
\filldraw[white] (2,0) circle (0.4cm) node {$w_1$};
\filldraw[white] (4,0) circle (0.4cm) node {$w_2$};
\draw (2,0) circle (0.4cm) node {$v_0$};
\draw (4,0) circle (0.4cm) node {$w$};
\end{tikzpicture}
\]
Since there are at least $b[g,\ell]-3$ simple branch points at $v_0$,   the component $v_{g-1}$ 
must map with full degree $d[g,\ell]$. Therefore,  one of the three markings at $w$
is forced 
to be the image of marking $p_{n[g,\ell]}$. 
Moreover, the point $p_{n[g,\ell]}$ must lie on the unique loop in $\Gamma$ which goes through the vertex over $w$ containing the two ramification points. We are in the case of Contribution 2.

\vspace{10pt}
\noindent \textbf{Case (iii)} The graph $\Gamma'$ has the following shape:
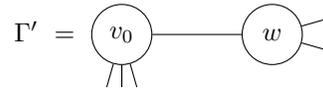
\[
\Gamma' \ = \ 
\begin{tikzpicture}[baseline = -0.1 cm]
\draw (2,0) to (4,0);
\draw (2,0) -- (2,-0.7);
\draw (2,0) -- (2,-0.7);
\draw (2,0) -- (2.2,-0.7);
\draw (2,0) -- (1.8,-0.7);
\draw (4,0) -- (4.7,0.2);
\draw (4,0) -- (4.7,-0.2);
\filldraw[white] (2,0) circle (0.4cm) node {$w_1$};
\filldraw[white] (4,0) circle (0.4cm) node {$w_2$};
\draw (2,0) circle (0.4cm) node {$v_0$};
\draw (4,0) circle (0.4cm) node {$w$};
\end{tikzpicture}
\]
The markings at $w$ are branch points, so the remaining $b[g,\ell]-2$ branch points are at $v_0$. As we have seen before, in case (iii), at most one of the associated ramification points can be on a component not equal to $v_{g-1}$, which still forces $b[g,\ell]-3$ ramification points on $v_{g-1}$.  The component $v_{g-1}$ must map with full degree $d[g,\ell]$. Hence we arrive at a contradiction, since the marking $p_{n[g,\ell]}$ lies over $v_0$ and thus in $v_{g-1}$. The curve can not have the  shape \eqref{eqn:Ccurveshape}.

\vspace{10pt}
 Contributions 1,2, and 3 are therefore the only nonvanishing terms in the intersection. After summing the three contributions to the degree of $\tau_{g,\ell,r}$ and dividing by $b[d,\ell]!$ , we obtain the recursion
 stated in the Proposition.
\end{proof}

\section{Proof of Theorems \ref{vvtt3}-\ref{vvtt7} for $\ell\geq 0$}

\subsection{The genus 0 case}
The recursion of Proposition \ref{recc}
reduces the calculation of  the Tevelev degree
$\tau_{g,\ell,r}$ to the genus $g=0$ case.
Indeed, for $n[g,\ell]-r+1 \geq 3$, we can apply the recursion and lower the genus and, for $n[g,\ell]-r+1 < 3$,  $\mathsf{Tev}_{g,\ell,r}$ is defined to be $0$.
In genus $0$, since
$$1 \leq r \leq d[0,\ell] = \ell+1\, $$
by \eqref{ffrd},
we must have $\ell\geq 0$.
The Tevelev degrees in genus 0
for $\ell\geq 0$ 
are determined 
by the following result.

\begin{proposition} \label{genus-0-case}
For $\ell \geq 0$ and $1\leq r \leq \ell +1$ we have 
$$ \mathsf{Tev}_{0,l,r}=1.$$
\end{proposition}

\begin{proof} Let $(\mathbb{P}^1, p_1,\ldots,p_n)$ be a  genus 0
curve with $n$ distinct markings, for
$$n=n[0,\ell]=3+2\ell \geq 3\, .$$
Let $\mathcal{M}_{d,r}$ be the moduli space of maps
$$\pi: (\mathbb{P}^1,p_1,\ldots,p_n) \rightarrow \mathbb{P}^1$$
of degree $d[0,\ell]=1+\ell$ satisfying
$$\pi(p_{n-r+1}) = \ldots = \pi(p_n)= [1,0]
\in \mathbb{P}^1\, .$$
The moduli space $\mathcal{M}_{d,r}$ is
the nonsingular $n-r$ dimensional quasi-projective subvariety 
$$\mathcal{M}_{d,r} \subset 
\mathbb{P}(H^0(\mathbb{P}^1,\mathcal{O}_{\mathbb{P}^1}(d)) \oplus H^0(\mathbb{P}^1,\mathcal{O}_{\mathbb{P}^1}(d)
))
\ $$ defined by  pairs of sections $(s_0,s_1)$
of $H^0(\mathbb{P}^1,\mathcal{O}_{\mathbb{P}^1}(d))$ satisfying:
\begin{enumerate}
    \item [(i)] $s_0$ and $s_1$ have no
    common zeros,
    \item[(ii)] $s_1$ vanishes at the $r$ points
    $p_{n-r+1}$,\ldots, $p_n$.
\end{enumerate}
Conditions (ii) are
linear, so 
$\mathcal{M}_{d,r}\subset \mathbb{P}^{2d+1-r}= \mathbb{P}^{n-r}$.

The tangent space of $\mathcal{M}_{d,r}$
at $[\pi]\in \mathcal{M}_{d,r}$
is 
$$ \mathsf{Tan}_\pi =H^0(\mathbb{P}^1, \pi^*(\mathsf{Tan}_{\mathbb{P}^1})(-p_{n-r+1}-\ldots-p_n))\, .$$ 
The degree of the line bundle
$\pi^*(\mathsf{Tan}_{\mathbb{P}^1})(-p_{n-r+1}-\ldots-p_n)$
is $$2(1+\ell) - r =n-r-1\, .$$ 
Since the
points $p_1,\ldots,p_{n-r}$ impose independent
conditions on sections of $\pi^*(\mathsf{Tan}_{\mathbb{P}^1})(-p_{n-r+1}-\ldots-p_n)$,
the differential
of the evaluation map
$$\text{ev}:\mathcal{M}_{d,r} \rightarrow (\mathbb{P}^1)^{n-r}\, , \ \ \ \ [\pi]\mapsto (\pi(p_1),\ldots,\pi(p_{n-r}))$$
is surjective.

Let $x_1\,\ldots,x_{n-r}\in \mathbb{P}^1$
be general points. For $1\leq i \leq n-r$, let
$$\mathsf{H}_i\subset \mathcal{M}_{d,r}$$
be the locus of maps $\pi$ satisfying
$$\pi(p_i)=x_i\, .$$
The closure 
$\overline{\mathsf{H}}_i\subset
\mathbb{P}^{2d+n-r}$ 
is a linear space of codimension 1.
Since 
\begin{equation}\label{xxkx}
\mathsf{ev}^{-1}(x_1,\ldots,x_{n-r})\subset \mathcal{M}_{d,r}
\end{equation}
is both transverse and linear,  $\mathsf{ev}^{-1}(x_1,\ldots,x_{n-r})$ is
either empty or a reduced point. Since the image of
$\text{ev}$ is dense, 
$\mathsf{ev}^{-1}(x_1,\ldots,x_{n-r})$ must be a reduced point.{\footnote{Alternatively,  the empty case is easily seen to be impossible.
The hyperplane closures
 have to intersect in the linear space 
$$\overline{\mathcal{M}}_{d,r} =\mathbb{P}^{n-r} 
\subset 
\mathbb{P}(H^0(\mathbb{P}^1,\mathcal{O}_{\mathbb{P}^1}(d)) \oplus H^0(\mathbb{P}^1,\mathcal{O}_{\mathbb{P}^1}(d)
)\, .$$
Using the standard description of
$$\partial\overline{\mathcal{M}}_{d,r} = \overline{\mathcal{M}}_{d,r}\setminus \mathcal{M}_{d,r}$$
via maps of lower degrees, the intersection
of $\overline{\mathsf{H}}_1\cap \ldots\cap
\overline{\mathsf{H}}_{n-r}$ is easily seen
to be disjoint from $\partial\overline{\mathcal{M}}_{d,r}$ using the
genericity of 
$q_1,\ldots, q_{n-r}$.}}

Finally, a dimension count shows
the single reduced point
\eqref{xxkx}  corresponds to
a map with only simple ramification points.
Therefore, $\mathsf{Tev}_{0,\ell,r}=1$.
\end{proof}

\subsection{Induction argument}
We are ready to completely calculate
the Tevelev degrees in case $\ell \geq 0$. Because of the 
form of the recursion \eqref{eqn:tevrecursion}, we 
consider all the cases with $\ell \geq 0$ at once instead of proving Theorems \ref{vvtt3}, \ref{vvtt6}, and \ref{vvtt7} separately.

\vspace{10pt}
\noindent{\bf Theorem.} 
{\em For $g \geq 0$, $ \ell \geq 0$ and $1 \leq r \leq \ell +1 +g$, we have}
$$
\mathsf{Tev}_{g,l,r}=
\begin{cases}
2^g- \sum_{j=0}^{r-2} {g \choose j} &\text{for } \ell =0,
\\  \mathsf{Tev}_{g,0,\max(r-l,1)} & \text{for } \ell \leq r,
\\ 2^g  & \text{for } \ell\geq r.
\end{cases}
$$

\begin{proof}
We proceed by induction on $g$.
The $g=0$ case is
covered by Proposition
\ref{genus-0-case}. Suppose $g > 0$. 
\begin{enumerate}[label=\underline{Step \arabic*}]
 \item We consider the case $ \ell \geq r \geq 1$. If we start with $ \ell \geq r$, the same inequality persists for the
 recursion, and we have
 $$
 \mathsf{Tev}_{g,\ell,r}= \mathsf{Tev}_{g-1,\ell,\max(1,r-1)}+ \mathsf{Tev}_{g-1,\ell+1,r+1}= 2^{g-1}+ 2^{g-1}=2^g\,.
 $$
 \item We consider case $0\leq \ell \leq r$. We have already checked the case $\ell =r$. Suppose $ \ell < r$. If we start with $ \ell < r$, then the non-strict inequality $ \ell \leq r$ persists in the first step of the recursion. We have
 \begin{align*}
\mathsf{Tev}_{g,\ell,r} 
& =\mathsf{Tev}_{g-1,\ell ,\max(r - 1,1)} + \mathsf{Tev}_{g-1,\ell+1,r+1}\\
& = 2^{g-1}- \sum_{j=0}^{ \max(r-3-l,-1 - \ell)  } { g-1 \choose j} + 2^{g-1} - \sum_{j=0}^{ r-2-l } { g-1 \choose j} \\
& = 2^g - \sum_{j=0}^{ r-2-l } \Bigg[ { g-1 \choose j-1} + {g-1 \choose j} \Bigg]\\ &= 2^g - \sum_{j=0}^{ r-2-l } { g \choose j}\ ,
 \end{align*}
 where in the first equality we are using the inductive step applied to $g$. 
 \end{enumerate}
 The proof is complete.
\end{proof}

\subsection{Lines in projective space} \label{Sect:Gale}
By the formulas of Theorems \ref{vvtt6} and \ref{vvtt7}, we have
\begin{equation}\label{vkke}
\mathsf{Tev}_{g,\ell,g+1+\ell}=1\, .
\end{equation}
The Tevelev degree is defined via the map
$$\tau_{g,\ell,g+1+\ell}: \overline{\mathcal{H}}_{g,g+1+\ell,g+3+2\ell,
g+1+\ell} \rightarrow
\overline{\mathcal{M}}_{g,g+3+2\ell}\times 
\overline{\mathcal{M}}_{0,3+\ell}\, .$$
There is a simple geometric reason for the
Tevelev degree to be 1 here.

Let $(C,p_1,\ldots,p_{g+3+2\ell})$ be a general
nonsingular pointed curve.
The last $g+1+\ell$ markings (associated to $r$) determine a line bundle $L$ on $C$. Consider 
the complete linear series $\mathbb{P}(H^0(C,L))$
of dimension $1+\ell$ by Riemann-Roch.

\vspace{5pt}
\noindent $\bullet$ There is a distinguished element
$$\zeta =[p_{3+\ell}+ \ldots +p_{g+3+2\ell}]
\in \mathbb{P}(H^0(C,L))\, .$$

\vspace{5pt}
\noindent $\bullet$ There are $2+\ell$ distinguished
hyperplanes
$$\mathsf{H}_1,\ldots, \mathsf{H}_{2+\ell} \subset \mathbb{P}(H^0(C,L))$$
determined by the points
$p_1,\ldots,p_{2+\ell} \in C$.

\vspace{5pt}
\noindent $\bullet$ There is a projective space $\mathbb{P}(\mathsf{Tan}_\zeta)$
of dimension $\ell$ parameterizing
lines 
$$\mathsf{P}\subset \mathbb{P}(H^0(C,L))$$ passing through $\zeta$.

\vspace{5pt}
For a general such line $\mathsf{P}$, we obtain
a $3+\ell$ pointed  curve of genus 0,
$$(\mathsf{P}, \zeta, \mathsf{P}\cap \mathsf{H}_1,\ldots, \mathsf{P}\cap\mathsf{H}_{2+\ell})\, .$$
The associated rational map
\begin{equation}\label{nn33x}
\mathbb{P}(\mathsf{Tan}_\zeta) 
\dashrightarrow
\mathcal{M}_{0,3+\ell}
\end{equation}
is easily seen to be birational (of degree 1). 
The degree is 
$\mathsf{Tev}_{g,\ell,g+1+\ell}$.

The map \eqref{nn33x} has the flavour of a Gale transformation.
In the $\ell=1$ case, a direct connection to the Gale transformation can be made. Whether the  birationality of
\eqref{nn33x} can be explained for all $\ell$ in
terms of the Gale transformation is a question
left to the reader.

\section{Path counting}
\label{f665}
\subsection{Notation}
We consider here the case $\ell \leq 0$ and prove Theorem \ref{n33r}. Before
presenting the proof, we briefly recall the notation of Section \ref{dd13}.

For
$\ell \leq 0$ and $r \geq 1 $, we have $$n[g, \ell]=3+g+2 \ell\, .$$ After imposing $n[g, \ell]-(r-1) \geq 3$, we obtain the condition
\begin{equation*}
g \geq r-2 \ell-1\, .
\end{equation*}
Following the notation of Section \ref{dd13},
$$g[ \ell, r]= r - 2 \ell -1\, . $$
We also have $$d[g[\ell,r],\ell]=g[\ell,r]+1+\ell =
r-\ell>0\,.$$ 

Define the infinite vectors $\mathsf{T}_{\ell, r}$ and $\mathsf{E}_r$ by
$$
\mathsf{T}_{\ell,r}[j]= \mathsf{Tev}_{g[\ell,r]+j, \ell,r}
\ \ \ \text{and} \ \ \
\mathsf{E}_{r}[j] = 2^{r+j-1} - \sum_{i=0}^{r-2}\binom{r+j-1}{i}\, 
$$
for $j \geq 0$. 
We set $$
\mathsf{T}_{\ell,r}[j]= 0 
\ \ \ \text{and} \ \ \ \mathsf{E}_{r}[j]=0
$$
for $j < 0$.
For all $j \in \Z$ and $r \geq 1$, we have
\begin{equation} \label{Pascal-type-law-for-E}
\mathsf{E}_{r+1}[j+1] = \mathsf{E}_{r+1}[j] +\mathsf{E}_{r}[j+1]\, .
\end{equation}

\subsection{Proof of Theorem \ref{n33r} for \texorpdfstring{$\ell\leq 0$}{l<=0}}
By Theorems \ref{vvtt6} and \ref{vvtt7} for $\ell \geq 0$, we already know  
$$
\mathsf{T}_{0,r}= \mathsf{E}_{r} =\sum_{\gamma\in \mathsf{P}(0,r)} 
\mathsf{E}_{\mathsf{Ind}(\gamma)}\ 
$$
holds.

We  prove the formula of Theorem \ref{n33r} for the coefficient $$\mathsf{T}_{\ell,r}[j] 
\ \  \text{for}\ \ 
j\geq 0$$ by induction on $g=g[\ell,r]+j$. When $g[\ell,r]+j=0$, we must have $(\ell,r)=(0,1)$ and $j=0$. 
So the base of the induction is established.

Suppose now 
$g[\ell,r]+j>0$. The proof of the induction step is split into two cases:

\vspace{5pt}
\noindent \textbf{Case 1.}
Suppose $r=1$ and $ \ell \leq 0$. 

\vspace{5pt}
\noindent $\bullet$ If $j=0$, we have 
\begin{align*}
\mathsf{T}_{\ell,1}[0]&= \mathsf{Tev}_{g[\ell,1],\ell,1}\\ &= \mathsf{Tev}_{g[\ell,1]-1,\ell+1,2} \\
&= \mathsf{T}_{\ell+1,2}[0] \\
&=\sum_{\gamma\in \mathsf{P}(\ell +1 ,2)} 
\mathsf{E}_{\mathsf{Ind}(\gamma)}[0] \\
&=\sum_{\gamma\in \mathsf{P}(\ell ,1)} 
\mathsf{E}_{\mathsf{Ind}(\gamma)}[0]\, .
\end{align*}
The second equality follows from \eqref{eqn:tevrecursion} since $$\mathsf{Tev}_{g[\ell,1]-1,\ell,1}=\mathsf{T}_{\ell,1}[-1]=0\, .$$
The last equality follows from    $\mathsf{E}_{\mathsf{Ind}(\gamma)}[0]=1$ for all $\gamma\in \mathsf{P}(\ell ,1) \cup \mathsf{P}(\ell +1 ,2)$ and the fact that $\mathsf{P}(\ell +1 ,2)$ and $\mathsf{P}(\ell ,1)$ are in bijection.

\vspace{5pt}
\noindent $\bullet$
If $j > 0$, we have 
\begin{align*}
 \mathsf{T}_{\ell,1}[j]&= \mathsf{Tev}_{g[\ell,1]+j,\ell,1}\\
 & =\mathsf{Tev}_{g[\ell,1]+j-1,\ell,1}+\mathsf{Tev}_{g[\ell,1]+j-1,\ell+1,2}\\
&=\mathsf{T}_{\ell,1}[j-1]+\mathsf{T}_{\ell+1,2}[j] \\
& =\sum_{\gamma\in \mathsf{P}(\ell,1)} 
\mathsf{E}_{\mathsf{Ind}(\gamma)}[j-1]+ \sum_{\gamma\in \mathsf{P}(\ell +1 ,2)} 
\mathsf{E}_{\mathsf{Ind}(\gamma)}[j] \\
& =\sum_{\gamma\in \mathsf{P}(\ell,1)} 
\mathsf{E}_{\mathsf{Ind}(\gamma)}[j-1]+ \sum_{\gamma\in \mathsf{P}(\ell ,1)} 
\mathsf{E}_{\mathsf{Ind}(\gamma)-1}[j] \\
& = \sum_{\gamma\in \mathsf{P}(\ell,1)} 
\mathsf{E}_{\mathsf{Ind}(\gamma)}[j]\, .
\end{align*}
In the fifth equality, we use the natural bijection between $\mathsf{P}(\ell +1 ,2)$ and $\mathsf{P}(\ell ,1)$ given by completing a path that ends in $(\ell +1 ,2)$ to a path that ends $(\ell ,1)$ by stepping with $\mathsf{D}$. In the last equality, we use \eqref{Pascal-type-law-for-E}.

\vspace{5pt}
\noindent \textbf{Case 2.} 
Suppose $r > 1$ and $\ell\leq 0$. 

\vspace{5pt}
\noindent For all $j \geq 0$, we have a similar chain of equalities:
\begin{align*}
\mathsf{T}_{\ell,r}[j]&= \mathsf{Tev}_{g[\ell,r]+j,\ell,r}\\
 & =\mathsf{Tev}_{g[\ell,r]+j-1,\ell,r-1}+\mathsf{Tev}_{g[\ell,r]+j-1,\ell+1,r+1} \\
&=\mathsf{T}_{\ell,r-1}[j]+\mathsf{T}_{\ell+1,r+1}[j]\\
& =\sum_{\gamma\in \mathsf{P}(\ell,r-1)} 
\mathsf{E}_{\mathsf{Ind}(\gamma)}[j]+ \sum_{\gamma\in \mathsf{P}(\ell +1 ,r+1)} 
\mathsf{E}_{\mathsf{Ind}(\gamma)}[j]\\
& = \sum_{\gamma\in \mathsf{P}(\ell,r)} 
\mathsf{E}_{\mathsf{Ind}(\gamma)}[j]\, ,
\end{align*}
which completes the induction step. \qed

\vspace{12pt}


\begin{table}[ht]
\footnotesize
\begin{tabular}
{c | c | c | c | c | c |}
\backslashbox{$\ell$}{$r$} & 1 & 2 & 3 & 4 & 5  \\
\hline
-1      
& $\mathsf{E}_{3}$    
& $2\mathsf{E}_3$
& $2\mathsf{E}_3+\mathsf{E}_4$
& \makecell{$2 \mathsf{E}_3+ \mathsf{E}_4$ \\ $+\mathsf{E}_5$}
& \makecell{$2 \mathsf{E}_3+ \mathsf{E}_4$
\\ + $ \mathsf{E}_5  +\mathsf{E}_6$} \\
\hline
-2      
& $2\mathsf{E}_4$     
& $2 \mathsf{E}_3 + 3 \mathsf{E}_4$
& \makecell{$4 \mathsf{E}_3+ 4 \mathsf{E}_4$ \\ $+ \mathsf{E}_5 $}
& \makecell{$6 \mathsf{E}_3 + 5 \mathsf{E}_4$ \\ $ + 2 \mathsf{E}_5+ \mathsf{E}_6$}
& \makecell{$8 \mathsf{E}_3+6 \mathsf{E}_4$ \\$+3 \mathsf{E}_5 +2 \mathsf{E}_6$ \\ $+\mathsf{E}_7$} \\
\hline
-3 
& $2 \mathsf{E}_4 + 3 \mathsf{E}_5$
& \makecell{$4 \mathsf{E}_3+6\mathsf{E}_4$ \\ $+4 \mathsf{E}_5$} 
& \makecell{$10 \mathsf{E}_3+11 \mathsf{E}_4$ \\ $+6 \mathsf{E}_5+\mathsf{E}_6$}
& \makecell{$18 \mathsf{E}_3 + 17 \mathsf{E}_4$ \\ $ + 9 \mathsf{E}_5 + 3 \mathsf{E}_6$ \\ $ + \mathsf{E}_7$}
& \makecell{$ 28 \mathsf{E}_3+ 24 \mathsf{E}_4$ \\ $+13 \mathsf{E}_5 + 6 \mathsf{E}_6$ \\ $ + 3 \mathsf{E}_7$} \\
\hline
-4 
& \makecell{$ 4 \mathsf{E}_4 + 6 \mathsf{E}_5$ \\  $+4 \mathsf{E}_6 $}
& \makecell{$10 \mathsf{E}_3+ 15 \mathsf{E}_4$ \\ $  +12 \mathsf{E}_5+5 \mathsf{E}_6$  }
& \makecell{$ 28 \mathsf{E}_3 + 32 \mathsf{E}_4$ \\  $+ 21 \mathsf{E}_5 + 8 \mathsf{E}_6$ \\ $+ \mathsf{E}_7$ } 
& \makecell{$56 \mathsf{E}_3 + 56 \mathsf{E}_4$ \\ $ + 34 \mathsf{E}_5 +14 \mathsf{E}_6$ \\ $ + 4 \mathsf{E}_7 + \mathsf{E}_8$}
& \makecell{$ 96 \mathsf{E}_3 + 88 \mathsf{E}_4$ \\ $ + 52 \mathsf{E}_5 + 24 \mathsf{E}_6$ \\ $ + 10 \mathsf{E}_7 + 4 \mathsf{E}_8$ \\ $ \mathsf{E}_9$ } \\
\hline
-5 
& \makecell{$10 \mathsf{E}_4 + 15 \mathsf{E}_5$ \\ $ + 12 \mathsf{E}_6 + 5\mathsf{E}_7 $ }
& \makecell{$28 \mathsf{E}_3 + 42 \mathsf{E}_4$ \\ $ + 36 \mathsf{E}_5 + 20 \mathsf{E}_6$ \\ $ + 6 \mathsf{E}_7 $ }
& \makecell{$ 84 \mathsf{E}_3 + 98 \mathsf{E}_4$ \\ $ + 70 \mathsf{E}_5 + 34 \mathsf{E}_6$ \\ $ +10 \mathsf{E}_7 + \mathsf{E}_8 $ }
& \makecell{$180 \mathsf{E}_3+ 186 \mathsf{E}_4$ \\ $ + 122 \mathsf{E}_5 + 58 \mathsf{E}_6$ \\ $ + 20 \mathsf{E}_7 + 5 \mathsf{E}_8$ \\ $ + \mathsf{E}_9$ }
& \makecell{$330 \mathsf{E}_3 +315 \mathsf{E}_4$ \\ + $198 \mathsf{E}_5 + 97 \mathsf{E}_6$ \\ $ + 40 \mathsf{E}_7 + 15 \mathsf{E}_8$ \\ $ + 5 \mathsf{E}_9 + \mathsf{E}_{10}$} \\
\hline
\end{tabular}
\caption{Values of $\mathsf{T}_{\ell,r}$ for $\ell<0$}
\end{table}

\subsection{Refined Dyck path counting} \label{rdpc}

For $\ell \leq 0$ and $r \geq 1$,
  we can write 
\begin{equation}\label{b445}
\mathsf{T}_{\ell,r}= \sum_{s=1}^{r- \ell +1} c_{\ell,r}^s \mathsf{E}_s
\end{equation}
by Theorem \ref{n33r} and Corollary \ref{Cor:TZcombination}.
By definition,
$$\mathsf{E}_1[j]=2^j\, .$$
For $s\geq 2$, define the vectors $\mathsf{E}'_s$ by
$$\mathsf{E}'_s [j]= \mathsf{E}_s[j] - 2^{s-1} \mathsf{E}_1[j]\, .$$
Then, 
$\mathsf{E}'_s[j]$ a polynomial in $j$ of degree $s-2$. 
The vectors $\mathsf{E}_s$ are
therefore independent, and
the coefficients $c^s_{\ell,r}$ in \eqref{b445}
are unique.

By Theorem \ref{n33r}, the
coefficient $c^s_{\ell,r}$
is the number of 
paths $\gamma\in \mathsf{P}(\ell,r)$
of index $s$ and therefore defines
a refined Dyck path counting problem.
The solution is given by the following result.

\begin{proposition} \label{Finding the c_s,l,r}

Let $ \ell \leq 0$ and $ r \geq 1$. 

\vspace{5pt}
\noindent
{\tiny{$\blacksquare$}}
If $\ell=0$ and $s \in \Z$,  then{\footnote{Here, $\delta^s_r$ is the Kronecker delta.}}
$
c_{0,r}^s= \delta^{s}_{r}$.

\vspace{5pt}
\noindent
{\tiny{$\blacksquare$}}
If $\ell=-1$ and $s \in \Z$, then
$
c_{-1,r}^s= \delta^{s -\delta_{r,1}}_{r+1}$.

\vspace{5pt}
\noindent
{\tiny{$\blacksquare$}}
If $\ell < -1$ and $3 \leq s \leq r- \ell +1$, then
\begin{equation*} 
c_{\ell,r}^s= \frac{(s-2)(s+r-4)}{|\ell| -2+r} {{|2 \ell| +r -s-1} \choose {|\ell| + 2- s} } + {{|2 \ell| +r-s-1} \choose {|\ell| -1}} -{{|2 \ell| +r-s-1} \choose {|\ell| + r-2}} \end{equation*}
and $c^s_{\ell,r}=0$ otherwise.
\end{proposition}

Before starting the proof, we require additional notation
and an auxiliary result. 
For $(u,v),(u',v') \in \mathcal{A}$ and $k \in \mathbb{Z}$, let 
$$
\mathcal{D}(k; u,v,u',v')
$$
be the set of paths $\gamma$
in $\mathcal{A}$ from $(u,v)$ to $(u',v')$ which take steps $U$ and $D$ and meet the $r=1$ axis
$$\{ \, (\ell,1)\, | \, \ell\leq 0\, \}\subset \mathcal{A}$$ 
exactly $k$ times. Let
$
d(k; u,v,u',v')= |\mathcal{D}(k; u,v,u',v')|
$.

\begin{lemma}\label{scott} Let $(u,v),(u',v') \in \mathcal{A}$. 
Let $t,j,k \in \mathbb{Z}$ satisfy $1 \leq t< k$, $u' \leq u$ and $j \leq -u$.  The following equalities then
hold:
\begin{enumerate}
    \item[(i)] $d(k;u+j,v,u'+j,v')=d(k;u,v,u',v')$ ,
    \item[(ii)] $d(k;0,1,u',v')=d(k;0,v',u'-v'+1,1)$ ,
    \item[(iii)] $d(k;u,v,u',v')=d(k-t;u,v+t,u',v')$ ,
    \item[(iv)] $d(k;0,1,u',1)=
\binom{-2 u' -k}{-u'-k+1} \frac{k-1}{-u'}$ .
\end{enumerate}

\end{lemma}
\begin{proof}
Identity (i) is trivial.
Identities (ii)-(iv) have been proven, for example, in \cite{Stack_Exchange} by B. Scott. We sketch here his proof adapted to our notation. 

\vspace{5pt}

\noindent $\bullet$ Identity (ii). Every
Dyck path in $\mathcal{A}$ from $(0,1)$ to $(u',v')$ can be written as a sequence of $-u'$ letters $D$ and $- u'+v'-1$ letters $U$. For any  such sequence, changing every $D$ with an $U$ and every $U$ with a $D$ yields a Dyck path in $\mathcal{A}$ from $(0,v')$ and $(u'-v'+1,1)$. Moreover, the number of intersections with the $r=1$ axis of the old and the new paths are the same. Therefore, we have bijection between $\mathcal{D}(k;0,1,u',v')$ and $\mathcal{D}(k;0,v',u'-v'+1,1)$.

\vspace{5pt}
\noindent $\bullet$ Identity (iii). Let $\gamma \in \mathcal{D}(k;u,v,u',v')$. For $1\leq i \leq t$, let $(\ell_i,1)$ be the first $t$ points at which $\gamma$ meets the $r=1$ axis. Here,
$$\ell_1 < \ldots < \ell_t\, .$$ 
Each of the points $(\ell_i,1)$ is immediately followed by a step $U$ in $\gamma$. Removing these $t$ steps $U$ and translating $\gamma$ upwards $t$ units yields a path in $\mathcal{D}(k-t;u,v+t,u',v')$. The operation is a bijection.

\vspace{5pt}
\noindent $\bullet$ Identity (iv). Using Identity (iii), we have 
\begin{align*}
d(k;0,1,u',1)&=d(1;0,k,u',1) \\
&= \binom {-2 u'-k}{-u'-k+1} - \binom{-2 u' -k}{-u'} \\
& = \binom{-2 u'-k}{-u'-1} \frac{k-1}{-u'}\, .
\end{align*}
The second equality is by directly 
counting the paths using a standard reflection trick (see \cite[p.22]{Comtet}). 
\end{proof}

\vspace{5pt}
\begin{proof}[Proof of Proposition \ref{Finding the c_s,l,r} ]
For $\ell=0,-1$ the result is trivial, so we
assume $\ell < -2$. 
We  write 
\begin{align*}
c_{\ell,r}^s& = | \{\gamma\in \mathsf{P}(\ell,r) : \mathsf{Ind}(\gamma)=s
\}| \\
& = \sum_{h=1}^{s-1} | \{\gamma\in \mathsf{P}(\ell,r) : | \gamma \cap \{ \ell=0\}|=h+1 \text{ and } | \gamma \cap \{ r=1\}|=s-h
\}| \\
&= \sum_{h=1}^{s-1} d(s-h-1;-1,h,\ell,r).
\end{align*}
By Lemma \ref{scott},  for $ h \in \{1,..., s-2 \}$,
the following chain of equalities holds:
\begin{align*}
d(s-h-1;-1,h,\ell,r)&= d(s-2;-1,1,\ell,r) \\
&=d(s-2;0,1,\ell+1,r) \\
&=d(s-2;0,r,\ell+2-r,1) \\
&=d(s+r-3;0,1,\ell+2-r,1) \\
& =\binom{|2 \ell| +r-s-1}{|\ell|-s+2} \frac{s+r-4}{|\ell|-2+r}\, .
\end{align*}
For $h=s-1$, a simple direct count  yields
\begin{align*}
d(0;-1,s-1,\ell,r) &=d(0;0,s-1,\ell+1,r) \\
&=\binom{|2 \ell| +r-s-1}{|\ell|-1} - \binom{|2 \ell| +r-s-1}{|\ell|-2+r}
\end{align*}
which is the number of allowed paths from $(0,s-2)$ to $(\ell+1,r-1)$. \\
The result follows by summing all the terms.
\end{proof}

\subsection{Fixed $j$}
An interesting  direction is to compute $\mathsf{T}_{\ell,r}[j]$
for fixed $j$ in terms of $\ell$ and $r$. The simplest example
is Castelnuovo's formula for $\ell<0$:
$$\mathsf{T}_{\ell,1}[0]= |\mathsf{P}(\ell,1)| = \frac{1}{|\ell|+1}
\binom{|2\ell|}{|\ell|}\,.$$

\begin{proposition}
For $ \ell <0$, we have 
$$
\mathsf{T}_{\ell,1}[1]=|\mathsf{P}(\ell,1)| + \frac{3}{ 2 | \ell | +1 } {{2 |\ell| + 1} \choose {|\ell | -1}} + \frac{4 (2 |\ell| -1)(2 | \ell | + 1) ({{2 |\ell| - 2} \choose {|\ell | -1}} - {{2 |\ell| - 2} \choose {|\ell |}}   )}{(|\ell| +1)(| \ell |+2)} \, .
$$
\end{proposition}

\begin{proof}
The case $l =-1 $ can be checked by hand, so we assume $ l < -1$. Using Theorem \ref{n33r}, we have
$$
\mathsf{T}_{\ell,1}[1]= \sum_{\gamma\in \mathsf{P}(\ell,1)} 
\mathsf{E}_{\mathsf{Ind}(\gamma)}[1]= |\mathsf{P}(\ell,1)| + \sum_{\gamma\in \mathsf{P}(\ell,1)} \mathsf{Ind}(\gamma)\, .
$$

For $ \gamma \in \mathsf{P}(\ell,1) $, define 
a {\em return step} of $\gamma$ to be a step $\mathsf{D}=(-1,-1)$ of $\gamma$ ending in 
the subset of $\mathcal{A}$ with $r=1$. Let $ r(\gamma)$ to be the number of return steps in $\gamma$. Then, 
\begin{equation}
    \label{x34f}
\mathsf{Ind}( \gamma) = r( \gamma) + |  \gamma \cap \{ \ell =0 \} |\, .
\end{equation}
By Section 6.6 of \cite{Deut}, the first term of \eqref{x34f} is
$$
\sum_{\gamma\in \mathsf{P}(\ell,1)} r ( \gamma) = \frac{3}{ 2 | \ell | +1 } {{2 |\ell| + 1} \choose {|\ell | -1}}\, .
$$
The second term of \eqref{x34f} is
\begin{align*}
\sum_{\gamma\in \mathsf{P}(\ell,1)} & |  \gamma  \cap \{ \ell = 0 \}| \\ 
& = \sum_{k=2}^{|\ell| + 1 } k \Bigg[{{2 |\ell| - (k-1)} \choose {|\ell |}} - {{2 |\ell| -(k-1)} \choose {|\ell | +1}}- \Bigg( {{2 |\ell| - k} \choose {|\ell | }} - {{2 |\ell| -k} \choose {|\ell | +1}}  \Bigg) \Bigg] \\
& = \sum_{k=2}^{|\ell| + 1 } k \Bigg[  {{2 |\ell| - k} \choose {|\ell |-1 }} - {{2 |\ell| - k} \choose {|\ell | }} \Bigg] \\
& =  \frac{4 (2 |\ell| -1)(2 | \ell | + 1) ({{2 |\ell| - 2} \choose {|\ell | -1}} - {{2 |\ell| - 2} \choose {|\ell |}}   )}{(|\ell| +1)(| \ell |+2)}\, .
\end{align*}
After summing the terms, we obtain the result.
\end{proof}

\section{Proof of Theorem \ref{negcase}}
\label{fccp}

As discussed at the end of Section \ref{dd13}, the formulas of Theorem \ref{negcase} 
cover the Tevelev degrees $\mathsf{T}_{g,\ell,r}$ in {\em all} cases. The master result
which specializes to Theorem \ref{negcase} and also covers the $\ell\geq 0$ cases takes
the following
form.

\begin{theorem}
Let $g \geq 0$, $\ell \in \mathbb{Z}$, and $r \geq 1$. 
\label{vvtt10}

\vspace{5pt}
\noindent
{\tiny{$\blacksquare$}}
If $g<r-2\ell-1$, then
\begin{equation} \label{eqn:Tevinequality} \mathsf{Tev}_{g,\ell,r} = 0\, .\end{equation}

\vspace{5pt}
\noindent  {\tiny{$\blacksquare$}}
If $g\geq r-2\ell-1$,
then the Tevelev degrees are

\vspace{5pt}
\noindent\ \ \ \ \ $\bullet$ for $r=1$,
\begin{equation} \label{eqn:Tevformula1}
    \mathsf{Tev}_{g,\ell,r} = 2^g -2 \sum_{i=0}^{-\ell-2} \binom{g}{i} + (-\ell  -2) \binom{g}{-\ell-1} + \ell \binom{g}{-\ell} \, ,
\end{equation}
 
 \vspace{5pt}
\noindent \ \ \ \ \ $\bullet$ for $r>1$,
\begin{equation} \label{eqn:Tevformular}
    \mathsf{Tev}_{g,\ell,r} = 2^g -2 \sum_{i=0}^{-\ell-2} \binom{g}{i} + (-\ell + r -3) \binom{g}{-\ell-1} + (\ell-1) \binom{g}{-\ell} - \sum_{i=-\ell+1}^{r-\ell-2} \binom{g}{i}\,.
\end{equation}
\end{theorem}

\begin{proof}
We prove the complete statment by induction on $g$. The main tool is the recursion of 
Proposition \ref{recc}.

We start with the genus 0 case. For $g=0$, 
 $\ell \geq 0$, and $1\leq r \leq \ell +1$,
we have
$$ \mathsf{Tev}_{0,l,r}=1$$
by Proposition \ref{genus-0-case}. In all other $g=0$ cases,
$$\mathsf{Tev}_{0,l,r}=0\, .$$ 
For $\ell <0$, the Theorem predicts $\mathsf{Tev}_{0,l,r}=0$ by \eqref{eqn:Tevinequality}. 
For $\ell \geq 0$ and $r=1$, formula \eqref{eqn:Tevformula1} specializes to 
$$1 + \ell \binom{0}{\ell} = 1\,.$$
Similarly, for $\ell \geq 0$ and $r > 1$, formula \eqref{eqn:Tevformular} specializes to
\[
1  + (\ell-1) \binom{g}{-\ell} - \sum_{i=-\ell+1}^{r-\ell-2} \binom{g}{i} = \begin{cases}
1-1=0 &\text{for }\ell=0\\
1-\sum_{i=-\ell+1}^{r-\ell-2} \delta_{i,0}  &\text{for }\ell>0\, .
\end{cases}
\]
In the $\ell>0$ case, the expression vanishes if and only if
$0 \leq r-\ell-2$ and returns $1$ otherwise, completing the verification of the case $g=0$.

\vspace{5pt}
Let $g > 0$. Then, $\mathsf{Tev}_{g, \ell,r}=0$ if
$$n[g,\ell] - r+1=g+3+2\ell-r+1 \ < \ 3\, .$$
Equivalently
$$g < r-2\ell-1\,, $$
recovering condition \eqref{eqn:Tevinequality}. If $g \geq  r-2\ell-1$,
we have the recursion
\begin{equation} \label{eqn:Tevrecursionrepeat}
    \mathsf{Tev}_{g,\ell,r} = \mathsf{Tev}_{g-1,\ell,\max(1,r-1)} + \mathsf{Tev}_{g-1,\ell+1,r+1}
\end{equation}
by Proposition \ref{recc}. 

After substituting the $g-1$ formulas (by induction), 
$\mathsf{Tev}_{g,\ell,1}$ is given by
\begin{align*}
  &  2^{g-1} -2 \sum_{i=0}^{-\ell-2} \binom{g-1}{i} + (-\ell + 1 -3) \binom{g-1}{-\ell-1} + \ell \binom{g-1}{-\ell} \\
+&2^{g-1} -2 \sum_{i=0}^{-\ell-3} \binom{g-1}{i} + (-\ell + 1 -3) \binom{g-1}{-\ell-2} + \ell \binom{g-1}{-\ell-1} - \underbrace{\sum_{i=-\ell}^{1-\ell-2} \binom{g-1}{i}}_{=0}\,.
\end{align*}
A straightforward check shows that the horizontally aligned terms combine to the terms in \eqref{eqn:Tevformula1} by the universal formula
\[
\binom{a-1}{b-1} + \binom{a-1}{b} = \binom{a}{b} \ \ \text{ for }a,b \in \mathbb{Z}\,.
\]
The check uses the formal manipulation
\begin{align*}
    \sum_{i=0}^{-\ell-3} \binom{g-1}{i} &= \sum_{i=0}^{-\ell-3} \binom{g-1}{i} + \underbrace{\binom{g-1}{-1}}_{=0} = \sum_{i=-1}^{-\ell-3} \binom{g-1}{i}\\
    &=\sum_{i=0}^{-\ell-2} \binom{g-1}{i-1}\,.
\end{align*}

Similarly, in case $r>1$, recursion
\eqref{eqn:Tevrecursionrepeat} shows that $\mathsf{Tev}_{g,\ell,r}$ is given by
\begin{align*}
    & 2^{g-1} -2 \sum_{i=0}^{-\ell-2} \binom{g-1}{i} + (-\ell + r -4) \binom{g-1}{-\ell-1} + (\ell-1) \binom{g-1}{-\ell} - \sum_{i=-\ell+1}^{r-\ell-3} \binom{g-1}{i}\\
    +& 2^{g-1} -2 \sum_{i=0}^{-\ell-3} \binom{g-1}{i} + (-\ell + r -3) \binom{g-1}{-\ell-2} + \ell \binom{g-1}{-\ell-1} - \sum_{i=-\ell}^{r-\ell-2} \binom{g-1}{i}
\end{align*}
The first two summands of each row combine to the corresponding terms in \eqref{eqn:Tevformular} as before. 

To combine the middle two summand pairs, we transfer a term $\binom{g-1}{-l-1}$ from the second to the first row. For the last summation, we use
\begin{align*}
    &\sum_{i=-\ell+1}^{r-\ell-3} \binom{g-1}{i} + \sum_{i=-\ell}^{r-\ell-2} \binom{g-1}{i}\\
   &\hspace{40pt}  =\sum_{i=-\ell+1}^{r-\ell-3} \binom{g-1}{i} + \sum_{i=-\ell}^{r-\ell-3} \binom{g-1}{i} + \binom{g-1}{r-\ell-2}\\
   &\hspace{40pt}  = \sum_{i=-\ell+1}^{r-\ell-2} \binom{g-1}{i} + \sum_{i=-\ell}^{r-\ell-3} \binom{g-1}{i}\\
    &\hspace{40pt}  =
    \sum_{i=-\ell+1}^{r-\ell-2} \binom{g-1}{i} + \binom{g-1}{i-1}\\
    &\hspace{40pt}  =
    \sum_{i=-\ell+1}^{r-\ell-2} \binom{g}{i}\,.
\end{align*}
We obtain precisely the last term of \eqref{eqn:Tevformular}, concluding the proof.
\end{proof}

If $g=-2\ell$, $\ell<0$, and $r=1$, 
Theorem \ref{vvtt10} yields
\begin{equation*} 
    \mathsf{Tev}_{-2\ell,\ell,1} = 2^{|2\ell|} -2 \sum_{i=0}^{|\ell|-2} \binom{|2\ell|}{i} + (|\ell| -2) \binom{|2\ell|}{|\ell|-1} - |\ell| \binom{|2\ell|}{|\ell|} \, ,
\end{equation*}
which equals the Catalan number $C_{|\ell|}$
as required by Castelnuovo's count.

\vspace{8pt}

\noindent Departement Mathematik, ETH Z\"urich\\
\noindent alessio.cela@math.ethz.ch

\vspace{8pt}

\noindent Departement Mathematik, ETH Z\"urich\\
\noindent rahul@math.ethz.ch

\vspace{8pt}

\noindent Mathematisches Institut der Universit\"at Bonn\\
\noindent schmitt@math.uni-bonn.de


\begin{thebibliography}{99}

\bibitem{ABC} N. Arkani-Hamed, J. Bourjaily, F. Cachazo, A. Postnikov, and J. Trnka,
{\em On-shell structures of MHV amplitudes beyond the planar limit}, J. of High Energy Physics {\bf 6} (2015), 179.

\bibitem{BaeSchmitt} Y. Bae and J. Schmitt, \newblock{\em Chow rings of stacks of prestable curves},
\href{https://arxiv.org/abs/2012.09887}{arXiv:2012.09887}. 


\bibitem{BP} P. Belorousski and
R. Pandharipande, {\em A descendent relation in genus 2}, Ann. Scuola Norm. Sup. Pisa Cl. Sci. {\bf 29} (2000), 171--191.

\bibitem{BM}
V. Bouchard and M. Mari\~{n}o,{\em Hurwitz numbers, matrix models and enumerative geometry} in {\em Hodge theory to integrability and TQFT tt*-geometry}, Proc. of Sympos. Pure Math. {\bf 78} (2008), 263--283.




\bibitem{Cast} G. Castelnuovo, {\em Numero delle involuzioni razionali giacenti sopra una
curva di dato genere}, Rendiconti R. Accad. Lincei {\bf 5} (1889), 130--133.

\bibitem{Comtet} L. Comtet,
{\em Advanced combinatorics, the art of finite and infinite expansions}, D. Reidel: Dordrecht and Boston (1974).

\bibitem{Dijk} R. Dijkgraaf,
{\em Mirror symmetry and elliptic curves}, The Moduli Space of Curves,
R. Dijkgraaf, C. Faber, G. van der Geer (editors), Progress in Mathematics, {\bf 129}, Birkh\"auser, 1995.

\bibitem{Deut} E. Deutsch, {\em Dyck path enumeration},
Discrete Math. {\bf 204} (1999), 167--202.

\bibitem{ELSV} T. Ekedahl, S. Lando, M. Shapiro, and A. Vainshtein,  {\em Hurwitz numbers and intersections on moduli spaces of curves}, Invent. Math. {\bf 146} (2001),  297--327.


\bibitem{FabP}
C.~Faber and R.~Pandharipande,
\newblock {\em Relative maps
and tautological classes},
\newblock {JEMS}, {\bf 7} (2005), 13--49.


\bibitem{GraberPandharipande}
T.~Graber and R.~Pandharipande,
\newblock {\em Constructions of nontautological classes on moduli spaces of curves},
\newblock {\em Michigan Math. J.}, {\bf 51} (2003), 93--109.

\bibitem{HM}
J. Harris and D. Mumford, 
{\em On the Kodaira dimension of the moduli space of curves}
With an appendix by William Fulton,
Invent. Math. {\bf 67} (1982),  23--88. 



\bibitem{Hur} A. Hurwitz,
{\em Ueber die Anzahl der Riemann'schen Fl\"achen mit gegebenen Verzweigungspunkten}, 
Math. Ann. {\bf 55} (1901), 53--66. 

\bibitem{JPPZ}
 F. Janda, R. Pandharipande, A. Pixton, D. Zvonkine, {\em Double ramification cycles on the moduli spaces of curves}, Publ. Math. Inst. Hautes Études Sci. {\bf 125} (2017), 221--266.


\bibitem{Lian2} C. Lian, {\em Non-tautological Hurwitz cycles}, \href{https://arxiv.org/abs/2101.11050}{arXiv:2101.11050}.

\bibitem{Lian1} C. Lian, {\em The H-tautological ring}, \href{https://arxiv.org/abs/2011.11565}{arXiv:2011.11565}.


\bibitem{OPcc} A. Okounkov and R. Pandharipande, {\em Gromov-Witten theory, Hurwitz numbers, and completed cycles}, Ann. of Math.
{\bf 163} (2006), 517--560.

\bibitem{OPmm} A. Okounkov and R. Pandharipande, {\em Gromov-Witten theory, Hurwitz numbers, and matrix models}, Proceedings of Algebraic geometry -- Seattle 2005, Proc. Sympos. Pure Math. {\bf 80}, Part 1, 324--414.



\bibitem{P} R. Pandharipande, {\em A geometric construction of Getzler's elliptic relation}, Math. Ann. {\bf 313} (1999), 715--729.

\bibitem{PHHT} R. Pandharipande and H.-H. Tseng,
{\em Higher genus Gromov-Witten theory of $\mathsf{Hilb}^n(\mathbb{C}^2)$ and CohFTs associated to local curves}, Forum Math. Pi {\bf 7} (2019).

\bibitem{Sage} SageMath, The Sage Mathematics Software System (Version 9.1),
   The Sage Developers, 2021, \url{https://www.sagemath.org}.

\bibitem{SvZ} J. Schmitt and J. van Zelm, {\em Intersections of loci of admissible covers with tautological classes}, Selecta Math. {\bf 26} (2020).

\bibitem{Stack_Exchange}
 B. Scott, \newblock{\em Number of Dyck paths that touch the diagonal exactly k times}, Mathematics Stack Exchange, \\ \url{https://math.stackexchange.com/questions/1900928/}. 


\bibitem{Stan} R. Stanley, {\em Enumerative combinatorics, Vol 2},
Cambridge Univ. Press: Cambridge, 1999.

\bibitem{Tev} J. Tevelev, {\em Scattering amplitudes of stable curves},
\href{https://arxiv.org/abs/2007.03831}{arXiv:2007.03831}.

\bibitem{Zelm}  J. van Zelm, {\em Nontautological bielliptic cycles}, Pacific J. Math. {\bf 294} (2018), 495--504.

\end{thebibliography}
\end{document}